\documentclass[11pt,letterpaper]{amsart}
\usepackage{graphicx}
\usepackage{amsmath, amssymb, amsthm, tikz, mathrsfs,verbatim,tensor}
\usepackage[left=1in,top=1in,right=1in]{geometry}
\usepackage{setspace}
\usepackage{enumerate}
\usepackage{amsfonts}
\usepackage{latexsym}
\usepackage{graphics}

\usepackage{bbm}

\newtheorem{theorem}{Theorem}[section]
\newtheorem{lemma}[theorem]{Lemma}
\newtheorem{corollary}[theorem]{Corollary}

\newtheorem{proposition}[theorem]{Proposition}
\newtheorem{definition}[theorem]{Definition}

\theoremstyle{definition}

\theoremstyle{plain}

\def\CS{\mathcal S}

\newcommand{\beqlbl}{\begin{equation}}
\newcommand{\eeqlbl}{\end{equation}}

\newcommand{\dyck}[1]{\text{Dyck}^{#1}}

\newcommand{\dn}{\dyck{2n}}
\newcommand{\sn}{S_n}

\newcommand{\R}{\ensuremath{\mathbb{R}} }

\newcommand{\N}{\ensuremath{\mathbb{N}} }

\newcommand{\Z}{\ensuremath{\mathbb{Z}} }
\newcommand{\D}{\ensuremath{\mathcal{D}} }

\newcommand{\A}{\mathcal{A}}

\newcommand{\cf}{\mathbf{1}}
\renewcommand{\P}{\mathbb{P}}

\newcommand{\E}{\mathbb{E}}

\newcommand{\fl}[1]{\lfloor #1 \rfloor}

\usepackage{hyperref}
\usepackage{tikz}

\tikzstyle{vertex}=[circle, draw, inner sep=0pt, minimum size=6pt]
\tikzstyle{Vertex}=[circle, draw, inner sep=0pt, minimum size=14pt]
\tikzstyle{Vertexc}=[circle, draw, inner sep=0pt, minimum size=14pt, fill=blue!30]
\tikzstyle{vertexc}=[circle, draw, inner sep=0pt, minimum size=6pt, fill=red!40]
\tikzstyle{vertexcg}=[circle, draw, inner sep=0pt, minimum size=6pt, fill=green!70!black]

\newcommand{\Vertex}{\node[Vertex]}

\newcommand{\Vertexn}{\node[]}

\newenvironment{pfofthm}[1]
{\par\vskip2\parsep\noindent{\em Proof of Theorem\ #1. }}{{\hfill
$\blacksquare$}
\par\vskip2\parsep}

\newcommand{\beq}{\begin{equation*}}
\newcommand{\eeq}{\end{equation*}}
\newcommand{\ba}{\begin{align*}}
\newcommand{\ea}{\end{align*}}

\newcommand{\matbegin}[1]{\left (  \begin{array} {#1} }
\newcommand{\matend}{ \end{array} \right ) }

\newcommand{\shortexc}{\text{SE}}

\newcommand{\bt}{\mathbf{t}}
\newcommand{\bT}{\mathbf{T}}
\newcommand{\Ht}{\mathrm{ht}}

\newcommand{\sA}{\mathscr{A}}

\title[Shapes and fluctuations for pattern-avoiding permutations]{Pattern-avoiding permutations and Brownian excursion Part I: Shapes and fluctuations}

\author[Christopher~Hoffman]{ \ Christopher~Hoffman$^\star\ddagger$}
\author[Douglas~Rizzolo]{ \ Douglas~Rizzolo$^{\star\dagger}$}
\author[Erik~Slivken]{ \ Erik~Slivken$^{\triangle \circ}$}

\thanks{\thinspace ${\hspace{-.45ex}}^\star$
Department of Mathematics,
University of Washington, Seattle, WA, 98195.
\hfill \\
\thinspace ${\hspace{-.45ex}}^\triangle$
Department of Mathematics,
University of California, Davis
Davis CA, 95616.
\hfill \\
${\hspace{-.45ex}}^\dagger$ 
Supported by NSF grant DMS-1204840
\hfil \\
${\hspace{-.45ex}}^\ddagger$ 
Supported by NSF grant DMS-1308645 and NSA grant H98230-13-1-0827
\hfil \\
${\hspace{-.45ex}}^\circ$ 
Supported by NSF RTG grant 0838212
\hfil \\
Email:
\hskip.02cm
\texttt{\{hoffman,drizzolo\}@math.washington.edu; erikslivken@math.ucdavis.edu}
}

\date{\today}

\begin{document}

\maketitle

\begin{abstract}
Permutations that avoid given patterns are among the most classical objects in combinatorics and have strong connections to many fields of mathematics, computer science and biology. In this paper we study the scaling limits of a random permutation avoiding
a pattern of length 3 and their relations to Brownian excursion.  
Exploring this connection to Brownian excursion allows us to strengthen the recent results of 
Madras and Pehlivan \cite{ML} and Miner and Pak \cite{mp}
as well as to understand many of the interesting phenomena that had previously gone unexplained. 
\end{abstract}

\section{Introduction}

\noindent

\smallskip

 Permutations are some of the most intensively studied objects in combinatorics.
One hundred years ago Percy MacMahon initiated the study of the class of pattern-avoiding permutations 
\cite{M}.  For $\pi \in \CS_k$ and $\tau\in \CS_n$, we say that $\tau$ contains the pattern $\pi$ if there exist $i_1<i_2<\cdots <i_k$ such that for all $1\leq r<s\leq k$ we have $\pi(r)<\pi(s)$ if and only if $\tau(i_r)<\tau(i_s)$.  We say $\tau$ avoids $\pi$, or is $\pi$-avoiding, if $\tau$ does not contain $\pi$. 
For example, a permutation $\tau \in \CS_n$ 
avoids the pattern $\textbf{321}$ if there exists no subsequence
$$1 \leq i_1 < i_2  < i_3 \leq n \qquad \text{ such that } \qquad
\tau(i_3)<\tau(i_2)<\tau(i_1).$$
  MacMahon showed that every \textbf{321}-avoiding permutation can be decomposed into two increasing subsequences and that the number of \textbf{321}-avoiding permutations of length $n$ is given by the $n$th Catalan number
$$C_n=\frac{1}{n+1}{2n\choose n}.$$

The modern study of pattern-avoiding permutations began with Donald Knuth who showed their importance  
in computer science. Knuth proved that the \textbf{231}-avoiding permutations are precisely those that can be sorted by a stack \cite{Knu}. He also showed that the number of $\textbf{231}$-avoiding permutations
is equal to $C_n$.
Further connections between pattern-avoiding permutations and sorting algorithms in computer science were explored by Tarjan, Pratt and others \cite{Pratt} \cite{Tarjan}.

From those starting points the study of pattern-avoiding permutations has gone in many directions.
One of the central  questions has been the exact enumeration of  $\CS_n(\pi)$, the set of permutations in $S_n$ that avoid $\pi$ \cite{gessel1990symmetric}.  Marcus and Tardos proved the Stanley-Wilf conjecture that says that $|\CS_n(\pi)|$ grows singly exponentially in $n$ for all $m$ and $\pi \in \CS_m$.  \cite{marcus2004excluded}. 
There are also strong connections to Kazhdan-Lusztig polynomials \cite{billey2001}, singularities of Schubert varieties \cite{elf}, Chebyshev polynomials  \cite{mansour} and  rook polynomials   \cite{babson}. Kitaev gives an extensive survey of the connections between pattern-avoiding permutations with other mathematical objects~\cite{Kit}.
In addition to their applications in computer science, pattern-avoiding permutations also are related to the partially asymmetric simple exclusion process model in statistical mechanics \cite{corteel2007tableaux} and the tandem duplication random-loss model in genome evolution \cite{bouvel2009variant}.

Another fundamental object in combinatorics is the statistics of uniform random permutations. These statistics have occupied a central place in both probability and combinatorics for hundreds of years since Montmort and Bernoulli showed that the distribution of the number of fixed points in a uniformly chosen random permutation is converging in distribution to a Poisson(1) random variable \cite{de1714essai}. A modern take on this classical result can be found in the study of  the cycle structure of a random permutation and other combinatorial stochastic processes \cite{pitmanbook}.

Finding the longest increasing subsequence of a random permutation is another probabilistic question that has generated extensive interest. 
This problem is part of the class of sub-additive processes which have typically proved intractible. This particular model has been home to a series of beautiful results starting with Vershik and Kerov who were able to show that the longest increasing subsequence in a permutation of length $n$ is typically on the order of $2\sqrt{n}$ \cite{vershik1981asymptotic} \cite{vershik1977asymptotics}.
This work was extended by Baik, Deift and Johansson who showed that the fluctuations around the mean converge to the Tracy-Widom distribution from random matrix theory \cite{baik1999distribution}.

Recently these two fundamental lines of research have merged with numerous results comparing statistics of a random pattern-avoiding permutation to the corresponding statistics of a uniformly random permutation.
The longest increasing subsequence of a pattern-avoiding permutation was studied in \cite{notknuth}
and the structure of the fixed points in pattern-avoiding permutations has been the subject many papers \cite{elizalde2, elizalde1, elizalde2004statistics, elizalde3, mp}.

In this paper we show fundamental connections between the shape of a random permutation avoiding a pattern of length 3 and Brownian excursion.  
These connections to Brownian excursion give our work a very geometric flavor. Glimpses of this geometric picture can be seen in the recent work of Janson \cite{Ja14}, Madras and Liu \cite{madras2010random}, Madras and Pehlivan \cite{ML} and Miner and Pak \cite{mp}. These papers study some analogs of the above results about uniformly random permutations, but they do not fully exploit this geometric point of view. The connections we find between pattern-avoiding permutations and Brownian excursion allow us to connect and vastly strengthen many previously disparate results concerning statistics of pattern-avoiding permutations. In particular we show that (in a sense that depends on the pattern) a random pattern-avoiding permutation converges to Brownian excursion.  In part II of this paper \cite{part2} we give very fine results about the distribution of fixed points. Hopefully these connections will serve as the building blocks of a unified theory of the structure of permutations avoiding a pattern of length three.

\subsection{Dyck paths and Brownian excursion}

All of our results are derived from bijections between Dyck paths and pattern-avoiding permutations.
Throughout the paper we use the following definition of a Dyck path.
\begin{definition}   A {\bf Dyck path} $\gamma$ is a sequence  $\{\gamma(x)\}_{x=0}^{2n}$ that satisfy the following conditions:

\begin{itemize}
\item $\gamma(0)=\gamma(2n)=0$
\item $\gamma(x) \geq 0$ for all $x \in \{0,1,\dots, 2n\}$ and
\item $|\gamma(x+1)-\gamma(x)| = 1$ for all $x \in \{0,1,\dots, 2n-1\}.$
\end{itemize}
\end{definition}

We often want to consider the function generated by a Dyck path through linear interpolation. Throughout this paper we often use the same notation to denote a sequence and the continuous function generated by extending it through linear interpolation.

Brownian excursion is the process 
 $\{\mathbbm{e}_t\}_{0 \leq t \leq 1}$  which is Brownian motion conditioned to be 0 at 0 and 1 and positive in the interior \cite{morters2010brownian}.
It is well known that the scaling limit of Dyck paths are Brownian excursion \cite{Ka76}  and that Dyck Paths of length 2n are in bijection with 
\textbf{321}-avoiding and \textbf{213}-avoiding permutations \cite{Knu, M}.

\subsection*{\textbf{321}-avoiding permutations}

\begin{figure} \label{introtau}
\centering
\begin{tikzpicture}
  \node [scale=.6, left] {\includegraphics{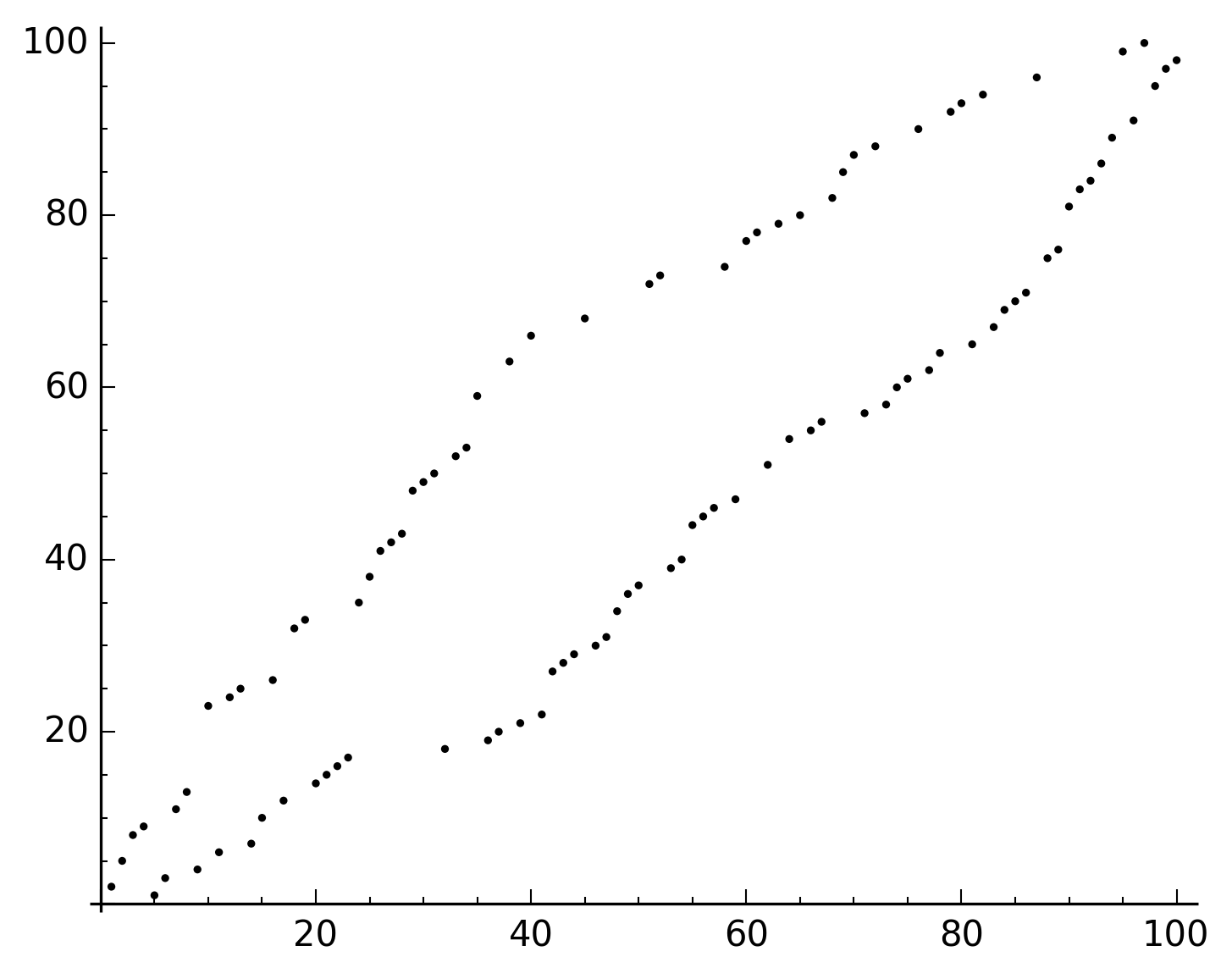}};
  \node [scale=.6,right] {\includegraphics{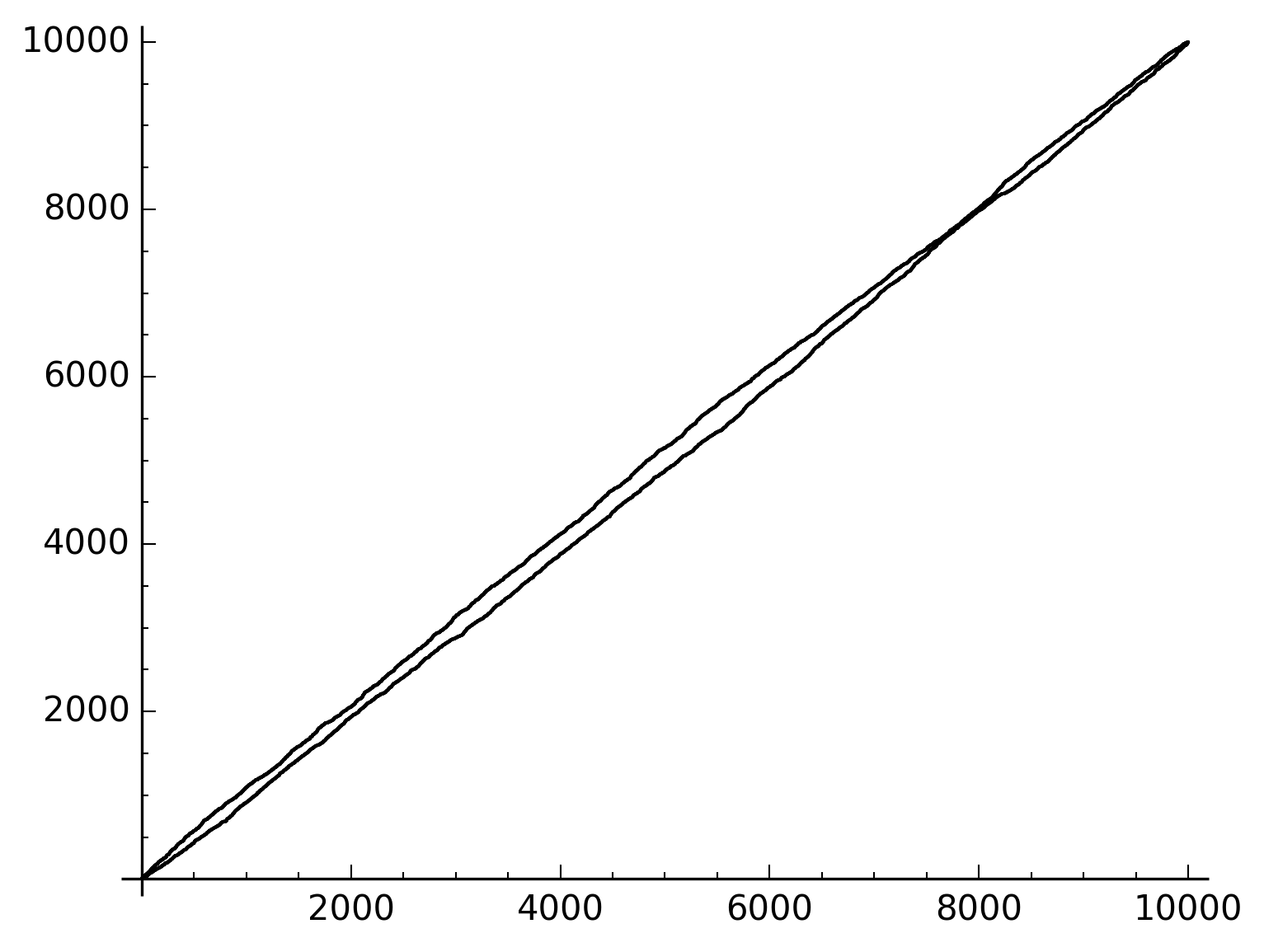}};
\end{tikzpicture}

\caption{ $\tau_{100}$ and $\tau_{10000}$ in $S_{100}(321)$ and $S_{10000}(321)$ respectively.} 
\end{figure}

MacMahon showed that \textbf{321}-avoiding permutation can be broken into two components that are (usually) of roughly equal size: a set of points above the diagonal and a set of points on or below the diagonal \cite{M}. 
From these two sequences we generate two functions. We prove that (properly normalized) both of them are converging in distribution to Brownian excursion. Moreover we show these two
functions are close to mirror images across the diagonal.

To show this we start with a Dyck path $\gamma$ of length $2n$. 
The Billey-Jockusch-Stanley bijection (described in Section \ref{offer}) gives us a permutation $\tau_{\gamma}$ which is \textbf{321}-avoiding \cite{BJS, callan}.  
For any permutation $\pi \in S_n$ define the exceedance process by
\begin{equation} \label{defense}E^n_{\pi}(i)=\pi(i)-i\end{equation}
and $E^n_{\pi}(0)=0$. We use the exceedance process to define several functions. 
For many of our results we define a subset of $A \subset \{0,1,2,\dots,n\}$ and then we define a function $F^A_{\pi}$ by linear interpolation through the points 
$$\left\{\left(\frac{a}{n}, \frac{E^n_\pi(a)}{\sqrt{2n}}\right)\right\}_{a \in A}.$$
For $\textbf{321}$-avoiding permutations $\tau_{\gamma}$ we define the functions
$F_{\tau_{\gamma}}^{E^+}(t)$ and $F_{\tau_{\gamma}}^{E^-}(t)$ 
where $E^+$ is the set of points where the exceedance process in \eqref{defense} is non-negative and $E^-$ is the set of points where the exceedance process is non-positive.
See the pictures in Figure \ref{introtau}.

\begin{figure}
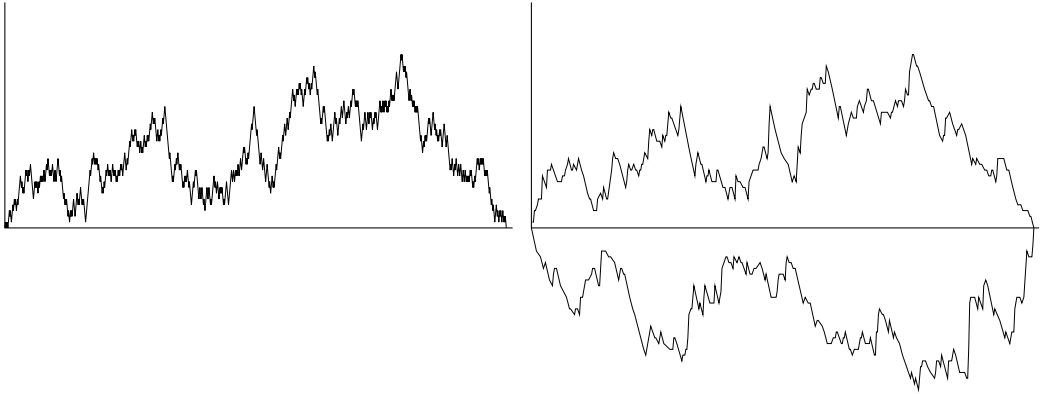
 

\centering

\caption{ $\Gamma^n(2nt)/\sqrt{2n}$ and the corresponding $F_{\tau_{\Gamma^n}}^{E^+}$ and $F_{\tau_{\Gamma^n}}^{E^-}$}
\label{introtau}

\end{figure}
\noindent

Our basic result about \textbf{321}-avoiding permutations is the following.

\begin{theorem} \label{cincodemayo}
Let $\Gamma^{n}$ be a uniformly chosen Dyck path of length $2n$ so $\tau_{\Gamma^n}$ is a uniformly chosen $\textbf{321}$-avoiding permutation. Then
$$\left(\frac{\Gamma^n(2nt)}{\sqrt{2nt}},F_{\tau_{\Gamma^n}}^{E^+}(t),F_{\tau_{\Gamma^n}}^{E^-}(t)\right)_{t\in [0,1]} \xrightarrow{\text{dist}}
(\mathbbm{e}_t,\mathbbm{e}_t,-\mathbbm{e}_t)_{t\in[0,1]},$$ 
where $(\mathbbm{e}_t,0\leq t\leq 1)$ is Brownian excursion and the convergence is in distribution on $C([0,1],\R^3)$.
\end{theorem}

The best previously known result along these lines was due to Miner and Pak who calculated the asymptotic distribution of $ \tau_\gamma(an)$ for all $a$, $0 \leq a \leq 1$ \cite{mp}.

\subsection*{\textbf{231}-avoiding permutations}

\begin{figure}
\centering
\begin{tikzpicture}
  \node [scale=.6,left] {\includegraphics{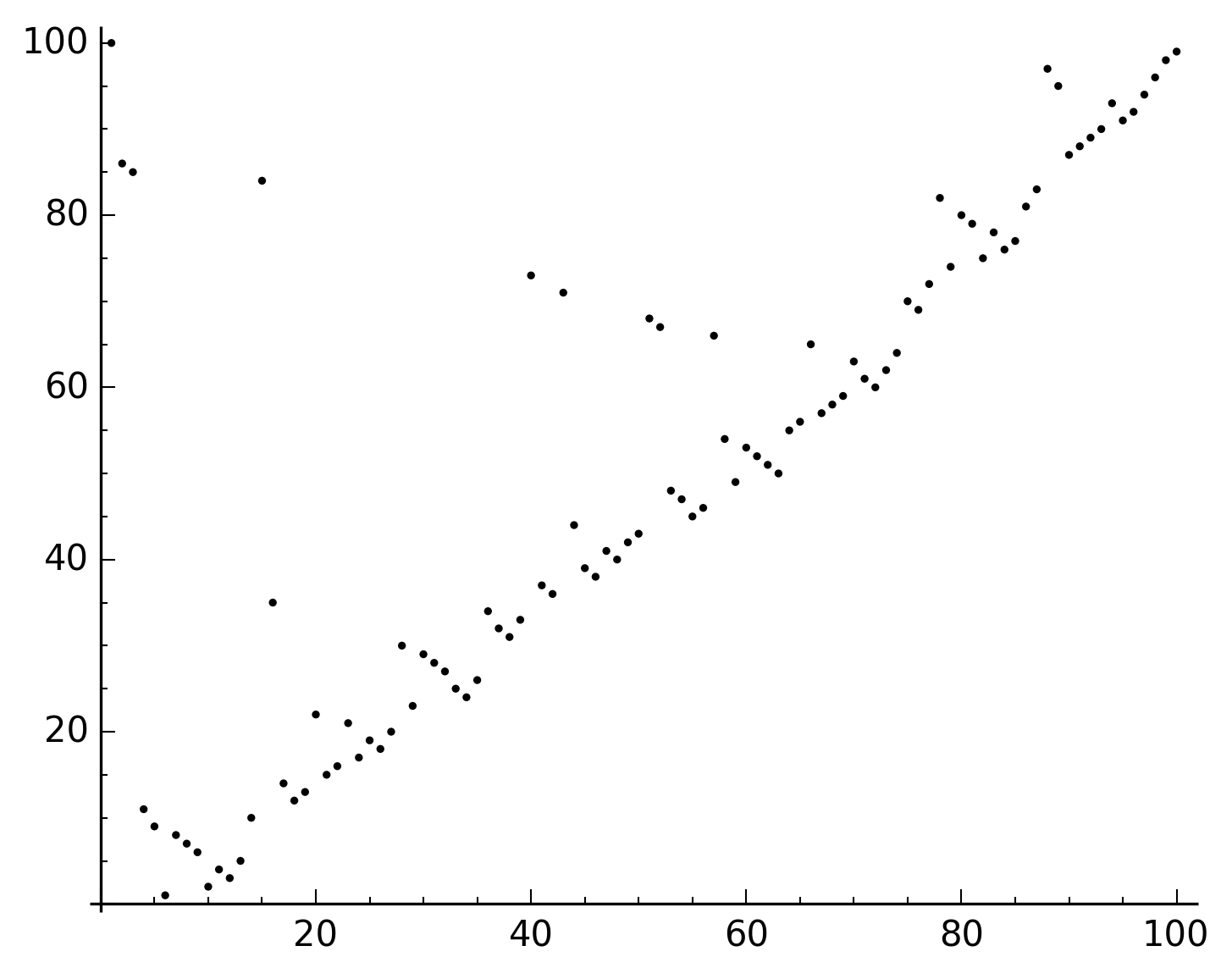}};
  \node [scale=.6,right] {\includegraphics{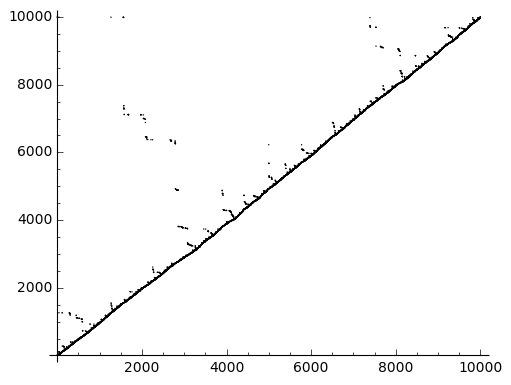}};
\end{tikzpicture}
\caption{ $\sigma_{100}$ and $\sigma_{10000}$ in $S_{100}(231)$ and $S_{10000}(231)$ respectively.} 
\end{figure}

In Section \ref{counteroffer} we define a bijection between Dyck paths and  \textbf{231}-avoiding permutations that allows us to explain the exceedance process in terms of specific geometric properties of Dyck paths.
For a Dyck path $\gamma$ we call its corresponding \textbf{231}-avoiding permutation $\sigma_{\gamma}$. 
Letting $\shortexc_\gamma$ consist of a set in $[n]$ where there are ``short excursions," we form the function $ F_{\sigma_{\gamma}}^{\shortexc{\gamma}}$ 
by linear interpolation and show that it very closely tracks $-\gamma(2nt)/\sqrt{2n}$ in the sup norm. Thus we get convergence to Brownian excursion. We also show that $\shortexc_\gamma$ is typically of size at least $n-n^{3/4+\epsilon}.$

\begin{figure}
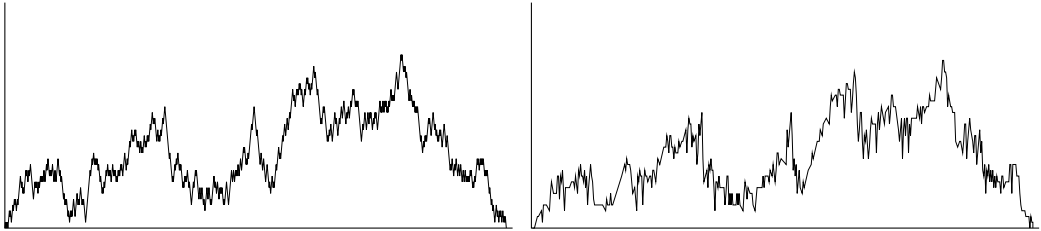
 

\centering

\caption{ $\Gamma^n(2nt)/\sqrt{2n}$ along with $-F_n$}
\end{figure}

\begin{theorem} \label{london}

Let $\Gamma^{n}$ be a uniformly chosen Dyck path of length $2n$ so $\sigma_{\Gamma^n}$ is a uniformly chosen $\textbf{231}$-avoiding permutation.
For any $\epsilon>0$ there exists a sequence of sets $\shortexc_{\Gamma^n}$ such that
$$\P\left(|\shortexc_{\Gamma^n}|>n-n^{.75+\epsilon}\right) \to 1$$ and 
$$ \left(\frac{\Gamma^n(2nt)}{\sqrt{2n}}, F_{\sigma_{\Gamma^n}}^{\shortexc_{\Gamma^n}}(t)\right)_{t\in [0,1]} \overset{dist}{\longrightarrow} 
\left( \mathbbm{e}_t,-\mathbbm{e}_t\right)_{t\in [0,1]},$$
the convergence being in distribution on $C([0,1],\R^2)$.
\end{theorem}

\section{321-avoiding permutations}
\label{offer}

We now describe a bijection (which is often known as the Billey-Jockusch-Stanley or BJS bijection) from Dyck paths of length $2n$ to \textbf{321}-avoiding permutations of length $n$ \cite{callan}.
Fix a Dyck path $\gamma:\{0,1,\dots, 2n\} \to\N$ of length $2n$.  Let $m$ be the number of runs of increases (or decreases) in $\gamma$. Let $a_i$ be the number of increases in the $i$th run of increases in $\gamma$. Let $A_i=\sum_{j=1}^ia_j$ and
let $\A = \cup_{i=1}^{m-1}\{A_i\}$ and $\bar \A=\{1,2,\dots,n\} \setminus (1+\A)$.
Similarly we define $d_i$ and $D_i=\sum_{j=1}^id_j$ based on the length of the descents. Then define 
$\D=\cup_{i=1}^{m-1}\{D_i\}$ and $\bar \D= \{1,2,\dots,n\} \setminus \D$.
We also set $A_0=D_0=0$.
Let $\tau_{\gamma}$ be the corresponding  \textbf{321}-avoiding permutation with the BJS bijection. This is defined by
$\tau_{\gamma}(D_i)=1+A_i$ on $\D$ and such that $\tau|_{\bar \D}=\bar \A$ 
and is increasing on $\bar \D$.

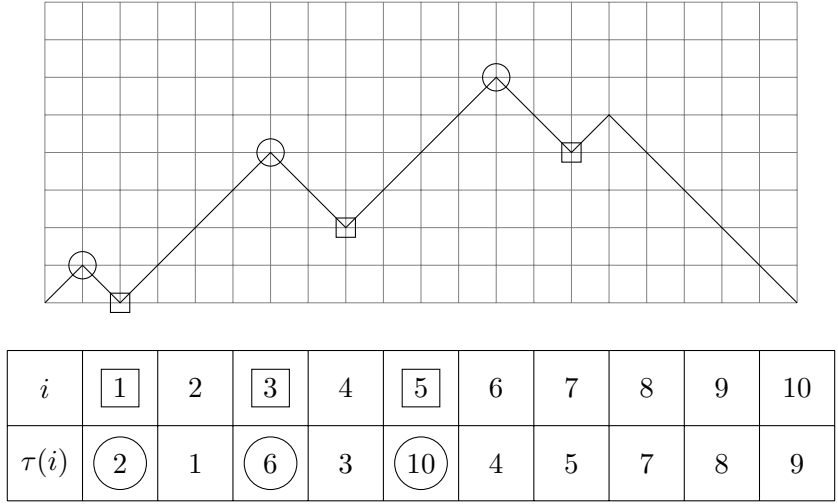
\begin{figure}
\centering
\begin{tikzpicture}

\draw [step=.5, help lines] (0,0) grid (10,4);

\draw (0,0)--(0, 0)--(1/2, 1/2)--(1, 0)--(3/2, 1/2)--(2, 1)--(5/2, 3/2)--(3, 2)--(7/2, 3/2)--(4, 1)--(9/2, 3/2)--(5, 2)--(11/2, 5/2)--(6, 3)--(13/2, 5/2)--(7, 2)--(15/2, 5/2)--(8, 2)--(17/2, 3/2)--(9, 1)--(19/2, 1/2)--(10, 0);

[place/.style={circle,draw,minimum size=6mm},
transition/.style={rectangle,draw,minimum size=4mm}]

\node at (3,2) [circle,draw]{};
\node at (.5,.5) [circle,draw]{};
\node at (6,3) [circle,draw]{};

\node at (1,0) [rectangle,draw] {};
\node at (7,2) [rectangle,draw] {};
\node at (4,1) [rectangle,draw] {};

\end{tikzpicture} 

\ \\

\begin{tikzpicture}
\draw [step=1,] (0,0) grid (11,2);

\node at (.5,.5) {$\tau(i)$};
\node at (.5,1.5) {$i$};
\foreach \i in {1,...,10}
	\node at (.5 + \i,1.5) {\i};	

\draw (1.25,1.25) rectangle (1.75,1.75);
\draw (3.25,1.25) rectangle (3.75,1.75);
\draw (5.25,1.25) rectangle (5.75,1.75);
\draw (1.5,.5) circle (.35cm);
\draw (3.5,.5) circle (.35cm);
\draw (5.5,.5) circle (.35cm);
\node at (1.5,.5) {2};
\node at (3.5,.5) {6};
\node at (5.5,.5) {10};

\node at (2.5,.5) {1};
\node at (4.5,.5) {3};
\node at (6.5,.5) {4};
\node at (7.5,.5) {5};
\node at (8.5,.5) {7};
\node at (9.5,.5) {8};
\node at (10.5,.5) {9};

\end{tikzpicture}
\caption{Dyck path of length 20 with correspond \text{\bf 321}-avoiding permutation.}
\end{figure}

For the rest of the section we let $y_0=0$ and 
\begin{equation}
\label{neal}
y_i=A_i-D_i=\gamma(A_i+D_i).
\end{equation}

\begin{lemma} \label{downandout}
For any $\gamma \in \dn$ and $j \in \{1,2,\dots,n\}$
$$\tau_{\gamma}(j)>j \ \ \ \text{ if } j \in \D$$
and
$$\tau_{\gamma}(j) \leq j \ \ \ \text{ if } j \not\in \D$$
\end{lemma}

 \begin{proof}
 If $j \in \D$ then there exists $i$ such that $D_i=j$ and
 $$\tau_{\gamma}(j)=\tau_{\gamma}(D_i)=1+A_i=1+A_i-D_i+D_i=1+y_i+D_i>D_i=j.$$
 If $j \not\in \D$ then there exists $i$ such that $D_{i-1}<j<D_i.$
 Note that $j$ is the $j-(i-1)$st element of $\bar \D$.
 As $\tau_{\gamma}$ maps $\bar \D$ monotonically to $\bar \A$ we get that
 $\tau_{\gamma}(j)$ is the $j-(i-1)$st element of $\bar \A$.
 There are at most $i-1$ elements of $\bar \A$ that are less than or equal to $j$. (Otherwise $1+A_i \leq j<D_i$ contradicting the non-negativity of $\gamma$.) Thus the $j-(i-1)$st element of $\bar \A$ is at most $j$ and $\tau_{\gamma}(j)\leq j$.
 \end{proof}
 
 We now undertake a more detailed analysis to show that for most \textbf{321}-avoiding permutations if $i,j$ are such that 
 $D_{i-1}<j< D_i$ then $$\tau_{\gamma}(j) \approx j-y_i$$ and if
 $j= D_i$ then $$\tau_{\gamma}(j) \approx j+y_i.$$
 Our first step is the following lemma.
\begin{lemma} \label{stanford}
Fix a Dyck path $\gamma \in \dn$ and $j \not \in \D$. There exists $i$ such that 
$D_{i-1}<j<D_i$. Then for $k \in\{1,\dots,m-1\}.$
\begin{enumerate}
\item If $A_k-(k-1)>D_i-i$ \hspace{.555in}then $\tau_{\gamma}(j)<A_k$.
\item If $A_k-(k-1)<D_{i-1}-(i-1)$ then $\tau_{\gamma}(j)>A_k+1$.
\end{enumerate}
\end{lemma}

\begin{proof}
Let $x=\max\bigg(\bar \A \cap\{1,2,\dots,A_k\}\bigg).$
In the first case we note that 
$$|\bar \A \cap\{1,2,\dots,A_k\}|=A_k-(k-1)>D_i-i=|\{1,2,\dots,D_i\}\cap \bar \D|.$$
Thus $$\tau_{\gamma}^{-1}(x)>D_i>j.$$
As $\tau_{\gamma}$ is monotone on the complement of $\D$ we get that 
$$A_k\geq x =\tau_{\gamma}(\tau_{\gamma}^{-1}(x))>\tau_{\gamma}(j).$$

In the second case
$$|\bar \A \cap\{1,2,\dots,A_{k}\}|=A_{k}-(k-1)<D_{i-1}-(i-1)
=|\{1,2,\dots,D_{i-1}\}\cap \bar \D|.$$
Thus $\tau_{\gamma}^{-1}(x)<D_{i-1}<j$ and
as $\tau_{\gamma}$ is monotone on the complement of $\D$
$$x = \tau_{\gamma}(\tau_{\gamma}^{-1}(x))<\tau_{\gamma}(j).$$
As $\tau_{\gamma}(j)>x$ and $\tau_{\gamma}(j) \in \bar \A$  we get that 
$$\tau_{\gamma}(j)>1+A_k.$$
\end{proof}

We now identify a class of moderate deviation properties of Dyck paths.
\begin{definition} \label{disparate}
We say that a Dyck path $\gamma \in \dn$ with associated sequences $A_i$ and $D_i$ satisfies the Petrov conditions if
\begin{enumerate}[(a)]
\item $\max_{x \in \{0,1,...,2n\}} \gamma(x)<.4n^{.6}$ \label{htcone}
\item $|\gamma(x)-\gamma(y)|<.5n^{.4}$ for all $x,y$ with $|x-y|<2n^{.6}$ \label{porter}

\item  $|A_i-A_j-2(i-j)|<.1|i-j|^{.6}$ for all $i,j$ with $|i-j|\geq n^{.3}$ and \label{dorsett}
\item  $|D_i-D_j-2(i-j)|<.1|i-j|^{.6}$ for all $i,j$ with $|i-j|\geq n^{.3}$ \label{tony}
\end{enumerate}
\end{definition}

\begin{lemma} \label{petrov}
With high probability the Petrov conditions are satisfied. The probability that they are not satisfied is decaying exponentially in $n^{c}$ for some $c>0$.
\end{lemma}
\begin{proof}
These results are standard Petrov style moderate deviation results \cite{Petrov1}. The general type of conditioning argument we need appears in \cite{MaMo03, PiRi13}. However we have not seen the exact results that we need anywhere in the literature so we include proofs of these statements in Appendix \ref{appendixa}.
\end{proof}

From these conditions we can derive many other moderate deviation results. We now list the ones that we will need.

\begin{lemma} \label{voucher}
If a Dyck path $\gamma \in \dn$ with associated sequences $A_i$ and $D_i$ satisfies the Petrov conditions then
$y_i <n^{.4}$ for all $i<n^{.6}$.

and for all $i>m-n^{.6}$.
$|A_i-A_{i-1}|, |D_i-D_{i-1}|<n^{.18}$ for all $i$.
This implies $|y_i-y_{i-1}|<n^{.18}$ for all $i$. 
Finally every consecutive sequence of length at least $n^{.3}$ has at least one element of $\D$ and at least one element of $\bar \D$.
\end{lemma}

\begin{proof}
Condition \eqref{porter} implies that the first claim.
Conditions (\ref{dorsett}) and (\ref{tony}) can be combined with a very inefficient use of the triangle inequality to show next set of claims. That every interval of length $n^{.3}$ has an element of $\D$ follows from the previous claim. Finally suppose there were an interval $(j,j+k)$ with $k\geq n^{.3}$ and no element of $\bar \D$. Then there exists an $i$ such that $D_i=j$ and $D_{i+k}=j+k$. This violates condition \eqref{tony}.
\end{proof}

\begin{lemma} \label{ducks}
For any Dyck path $\gamma \in \dn$ and any $j$ such that $D_{i-1}<j<D_i$ we get the following.
If the Petrov conditions are satisfied then

$$|\tau_{\gamma}(j)-j+y_i|<7n^{.4}$$
\end{lemma}

\begin{proof} 
We break the proof up into two cases. 
First consider the case that $y_i \geq 2n^{.4}$. 
Set $k^-=i-y_i-\lfloor n^{.4}\rfloor $ and $k^+=i-y_i+ \lfloor n^{.4} \rfloor $.
(Note that by Lemma \ref{voucher} if $y_i \geq 2n^{.4}$ then $i>n^{.6}$ and 
 $i<n-n^{.6}.$ Thus by Petrov condition \eqref{htcone} we have $0<k^-<k^+<n$.)
If the Petrov conditions are satisfied then 
\begin{equation}\label{bryant} n^{.4}\leq i-k^+<i-k^- < y_i+n^{.4}+1<.5n^{.6}.\end{equation}
Then by the Petrov condition \eqref{dorsett} and the definition of $y_i$
\begin{eqnarray}
D_i-i-A_{k^+}+k^+-1&\leq&A_i-y_i-i-A_{k^+}+k^+-1\nonumber\\
&\leq& A_i-A_{k^+}-(i-k^+)-y_i-1\nonumber\\
&\leq& (i-k^+)+(i-k^+)^{.6}-y_i-1\nonumber\\
&\leq& y_i-n^{.4}+(.5n^{.6})^{.6}-y_i-1\nonumber\\
&<&0. \label{progressive}
\end{eqnarray}

Again by \eqref{bryant}, Petrov conditions \eqref{dorsett}, Lemma \ref{voucher} and the definition of $y_{i-1}$
\begin{eqnarray}
D_{i-1}-(i-1)-A_{k^-}+k^--1
&\geq& A_{i-1}-y_{i-1}-(i-1)-A_{k^-}+k^--1\nonumber\\
&\geq& A_{i-1}-A_{k^-}-(i-1-k^-)-y_{i-1}-1\nonumber\\
&\geq& (i-1-k^-)-(i-1-k^-)^{.6}-y_{i-1}-1\nonumber\\
&\geq& y_i+n^{.4}-(.5n^{.6})^{.6}-y_{i-1}-2\nonumber\\
&\geq& y_i-y_{i-1}+.5n^{.4}\nonumber\\
&>&0. \label{geico}
\end{eqnarray}
By (\ref{progressive}), (\ref{geico}) and Lemma \ref{stanford}
$$A_{k^-}<\tau_{\gamma}(j)<A_{k^+}.$$
So by (\ref{bryant}) and Petrov condition \eqref{tony}
\begin{eqnarray*}
\tau_{\gamma}(j)-j
&>&A_{k^-}-D_i\\
&\geq& -D_i +D_{k^-}+y_{k^-}\\
&\geq& -2(y_i+n^{.4})-(i-k^-)^{.6}+y_{i}-n^{.4}\\
&\geq&-y_i-4n^{.4}.
\end{eqnarray*}
Again by (\ref{bryant}) and Petrov condition \eqref{tony}
\begin{eqnarray*}
\tau_{\gamma}(j)-j
&<&A_{k^+}-D_{i-1}\\
&\leq&D_{k^+}+y_{k^+}-D_{i-1}\\
&\leq&-2(k^+-(i-1))-y_i+n^{.4}\\
&\leq&-2y_i+2n^{.4}+(i-k^+)^{.6}-y_i+n^{.4}\\
&\leq&-y_i+4n^{.4}.
\end{eqnarray*}
Putting it all together we get
$$|\tau_{\gamma}(j)-j+y_i|<4n^{.4}.$$

Now consider the case that $y_i<2n^{.4}$. 
Notice that $D_i \leq A_i$ for all $i$. 

By Lemma \ref{downandout} we have that for all $j \not\in \D$  that $\tau_{\gamma}(j)\leq j$ and 
$$ \tau_{\gamma}(j)-j+y_i<2n^{.4}.$$ To get the lower bound we set 
$k^-=\max(0,i-4\lfloor n^{.4}\rfloor).$ 
If $i \leq 4n^{.4}$ and $k^-=0$ then by the Petrov bounds $j\leq D_i\leq 5n^{.4}$ and
$$\tau_{\gamma}(j)-j+y_i\geq 0-2n^{.4}-5n^{.4}=-7n^{.4}.$$

Finally we consider the case that $i \geq 4n^{.4}$. As $y_i\leq 2n^{.4}$ then by Lemma \ref{voucher} 
$y_{i-1}<3n^{.4}$. Again by the Petrov condition \eqref{dorsett}
\begin{eqnarray*}
D_{i-1}-(i-1)-A_{k^-}+k^--1 & \geq & A_{i-1}-y_{i-1} -(i-1)-A_{k^-} +k^- -1\\
& \geq & A_{i-1}-A_{k^-}-(i-1-k^-)+1-y_{i-1}\\
& \geq & 4n^{.4}-2 -(4n^{.4})^{.6}-y_{i-1}\\
&\geq & 0.
\end{eqnarray*}

We now finish this case exactly as we finished the case that $y_i>2n^{.4}$.
\end{proof}

\begin{lemma} \label{fathersday}
For any Dyck path $\gamma \in \dn$ that satisfies the Petrov conditions and any $j=D_i \in \D$
$$|\tau_{\gamma}(j)-j-\gamma(2j)|<10n^{.4} .$$
Also for any such $\gamma$, $j$ and $i$ with $D_{i-1}<j<D_i$
$$|\tau_{\gamma}(j)-j+\gamma(2j)|<10n^{.4} .$$
\end{lemma}

\begin{proof} If $j=D_i$ then by Petrov condition \eqref{htcone}
$$0\geq 2j-(A_i+D_i)=2D_i-A_i-D_i=D_i-A_i>-n^{.6}$$
so by the Petrov condition \eqref{porter}
\begin{equation}
\label{trojans}
|\gamma(2D_i)-\gamma(A_i+D_i)|=|\gamma(2j)-\gamma(A_i+D_i)|<n^{.4}.\end{equation}
Then by \eqref{trojans}
\begin{eqnarray*}
|\tau_{\gamma}(j)-j-\gamma(2j)|
&=&|1+A_i-j-\gamma(2j)|\\
&=&|1+D_i+\gamma(A_i+D_i)-j-\gamma(2j)|\\
&=&|1+\gamma(A_i+D_i)-\gamma(2j)|\\
&\leq& 1+n^{.4}\\
&\leq & 2n^{.4}.
\end{eqnarray*}

If $D_{i-1}<j<D_i$ then by Lemma \ref{voucher}
 $$0\geq 2j-2D_i\geq 2D_{i-1}-2D_i\geq -n^{.18}$$
and by Petrov condition \eqref{porter} for large $n$
\begin{equation} \label{crazy}
|\gamma(2j)-\gamma(2D_i)|<2(6n^{.18})^{.6}<6n^{1.2}<n^{.4}
\end{equation}
By Lemma \ref{ducks}

\begin{equation}|\tau_{\gamma}(j)-j+\gamma(A_i+D_i)|=|\tau_{\gamma}(j)-j+y_i|<7n^{.4}.\label{alameda}\end{equation} 
Then by the triangle inequality, (\ref{alameda}), (\ref{trojans}) and  \eqref{crazy}
\begin{eqnarray*}
|\tau_{\gamma}(j)-j+\gamma(2j)|
&<&|\tau_{\gamma}(j)-j+\gamma(A_i+D_i)|+|-\gamma(A_i+D_i)+\gamma(2D_i)|\\
&&+|\gamma(2j)-\gamma(2D_i)|\\
&<&7n^{.4}+n^{.4}+n^{.4}\\
&=&9n^{.4}.
\end{eqnarray*}
\end{proof}

\begin{pfofthm}{\ref{cincodemayo}}
Fix $0\leq a<b \leq 1$. By Lemma \ref{voucher} both $\frac{1}{\sqrt{2n}}\D$ and $\frac{1}{\sqrt{2n}}\bar \D$ have points in  $(a,b)$. Thus the theorem follows from Lemma \ref{fathersday} and the convergence of Dyck paths to Brownian excursion.

\end{pfofthm}

\section{A Bijection between Dyck Paths and 231-avoiding permutations} 
\label{counteroffer}

  We define a particular bijection from Dyck paths of length $2n$ to {\bf 231}-avoiding permutations of size $n$ which connects geometric properties of the path to geometric properties of the graph of the permutation.  Although we would not be surprised to learn that this bijection has appeared previously in the literature, we have not found it elsewhere.  For our purposes the most important geometric aspect of a Dyck path is an excursion.
\begin{definition}
An {\bf excursion} in a Dyck Path starting at $x$ with height $h$ and length $l$ is a path interval $\gamma([x,x+l])$ such that
\begin{enumerate}
\item $\gamma(x)=\gamma(x+l)=h-1$ 
\item $\gamma(x+1)=\gamma(x+l-1)=h$ and
\item $l=\min\{j\geq1:\ \gamma(x+j)=h-1\}.$
\end{enumerate}
\end{definition}

Note that there are $n$ excursions in a Dyck Path of length $2n$ as there is one excursion that begins with every up-step.  Based on this correspondence we say the $i$th excursion, $Exc(i)$ is the one that begins with the $i$th up-step.

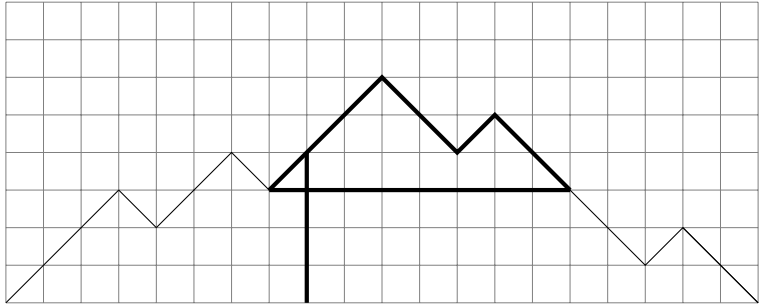
\begin{figure}
\centering

 \begin{tikzpicture} 
 \draw[step=.5, help lines] (0,0) grid (10,4);
 \draw (0,0) --(1.5,1.5);
 \draw (1.5,1.5)--(2,1);
 \draw (2,1)--(2.5,1.5)--(3,2)--(3.5,1.5) --(5,3)--(6,2)--(6.5,2.5)--(8.5,.5)--(9,1)--(10,0);
 \draw[ultra thick] (3.5,1.5) --(5,3)--(6,2)--(6.5,2.5)--(7.5,1.5);
 \draw[ultra thick] (4,0)--(4,2);
 \draw[ultra thick] (3.5,1.5) -- (7.5,1.5);

 \draw (9,1)--(10,0);
\end{tikzpicture}
\caption{A Dyck path in $\mathcal{D}_{10}$ with $v_6 = 8$, $h_6 = 4$, and $l_6 = 8$.}
 \label{fig.dyckex1}
\end{figure}

\begin{definition} \label{def.xihili}
For a Dyck path $\gamma$, define the following:

\begin{itemize}
\item $Exc(i):=$ the $i$th excursion.
\item $v_i:=$ the position after the $i$th up-step, or 1 + the start of $Exc(i).$
\item $h_i:=\gamma(v_i) = $ the height of the path after the start of $Exc(i).$
\item $l_i:=$ the length of the same excursion.  
\end{itemize}
\end{definition}

Figure \ref{fig.dyckex1} illustrates these definitions for a particular $\gamma.$

For a path $\gamma\in \dn$ we define the map $\sigma_\gamma = \sigma: [n] \to \Z$ by $$\sigma(i) = i + l_i/2  - h_i.$$

\begin{theorem} \label{bijection} For $\gamma \in \dn$ let $\sigma = \sigma_\gamma$ be defined as above.  Then $\sigma\in \sn(231).$  Moreover, $\gamma\mapsto \sigma_\gamma$ is a bijection from $\dn \to \sn(231).$
\end{theorem}
We break the proof of this theorem up into parts.

\begin{lemma}
For any Dyck path $\gamma \in \dn$,  $\sigma_\gamma:[n] \to [n].$  
\end{lemma}
\begin{proof}
Let $u_i$ and $d_i$ denote the number of up-steps and the number of down-steps, respectively, up until the beginning of the $i$th excursion and let $v_i=u_i+d_i$ denote the total number of steps until the beginning of the same excursion.  Each up-step is the beginning of an excursion so $u_i = i.$  Moreover $d_i$ is determined by $h_i$ since $h_i = u_i - d_i.$  For any path $\gamma\in\dn,$ $\gamma(x)\leq x.$  Therefore $1\leq h_i = \gamma(v_i) \leq i,$ hence $0\leq i - h_i.$  Moreover, $l_i/2$ counts the number of up-steps in the excursion.  Only $n-i$ up-steps remain after the first $i$ have occured so $l_i/2 -1 + i \leq n$.  Combining these inequalities gives: $$1\leq i-h_i +1  \leq \sigma(i) \leq i +l_i/2 - 1 \leq n.$$  Hence $\sigma$ maps $[n]$ into $[n]$. 
\end{proof}

\begin{lemma}
For any Dyck path $\gamma$ and any $i<j$ either 
$$Exc(j)\subset Exc(i) \ \ \ \text{ or } \  \ \ Exc(i) \cap Exc(j) = \emptyset.$$
\end{lemma}
\begin{proof}
This follows from the definition of an excursion.
\end{proof}
\begin{lemma}
For any Dyck path $\gamma$ and any $i<j$  if
$$Exc(j)\subset Exc(i) \ \ \ \text{ then } \  \ \  \sigma_{\gamma}(j)<\sigma_{\gamma}(i)$$ and if
$$Exc(j)\cap Exc(i) =\emptyset \ \ \ \text{ then } \  \ \  \sigma_{\gamma}(i)<\sigma_{\gamma}(j).$$
\end{lemma}

\begin{proof}
Let $1\leq i < j \leq n.$  By the previous lemma the $j$th excursion begins either before or after the $i$th excursion ends.  In other words $j-i < l_i/2$ or $j-i \geq l_i/2.$  \\

We first consider when $j-i < l_i/2$.  For $j$ in this region we have $h_j > h_i$ and $l_j<l_i.$    Moreover $l_j/2-l_i/2 < i-j.$  Therefore 

\beqlbl  \label{insideexcursion}
\sigma(j)-\sigma(i) = j-i + l_j/2 - l_i/2 - (h_j - h_i) \leq h_i - h_j <0.
\eeqlbl\\

\begin{figure}[h]
\centering
 \begin{tikzpicture} 
 \draw[step=.25, help lines] (0,0) grid (10,4);
 \draw (0,0) --(1,1)--(1.25,.75)--(2,1.5)--(2.5,1)--(2.75,1.25)--(3,1.5)--(3.5,2)--(3.75,1.75)--(4.5,2.5)--(5,2)--(5.25,2.25)--(6.5,1)--(6.75,1.25)--(7,1)--(7.75,1.75)--(8.25,1.25)--(8.5,1.5)--(10,0);
 \draw[thick] (2.75,0)--(2.75,1.25);
\draw[thick] (2.5,1)--(6.5,1);
 \draw[thick] (4,0)--(4,2);
 \draw[thick] (3.75,1.75) --(5.75,1.75);
 \draw (2.5,.5) node {$h_i$};
 \draw (5,.65) node {$l_i$};
 \draw(3.75,.8) node {$h_j$};
 \draw (4.9,1.45) node {$l_j$};
 
\end{tikzpicture}
\caption{The $j$th excursion occurs during the $i$th excursion}
\end{figure}

Now we consider when $j-i\geq l_i/2.$  Since the path must return below $h_i$ at the end of $Exc(i)$ then it needs at least $\max(0,h_j-h_i)$ up-steps after the the $i$th excursion ends to be at height $h_j$.  Therefore $j-i\geq l_i/2  + \max(0,h_j-h_i).$  This gives

\beqlbl \label{afterexcursion}
\sigma(j)-\sigma(i) = j-i + l_j/2-l_i/2 - (h_j-  h_i) \geq 1+ l_j/2 \geq 1.
\eeqlbl \\

\begin{figure}[h]
\centering
 \begin{tikzpicture} 
 \draw[step=.25, help lines] (0,0) grid (10,4);
 \draw (0,0) --(1,1)--(1.25,.75)--(2,1.5)--(2.5,1)--(2.75,1.25)--(3,1.5)--(3.5,2)--(3.75,1.75)--(4.5,2.5)--(5,2)--(5.25,2.25)--(6.5,1)--(6.75,1.25)--(7,1)--(7.75,1.75)--(8.25,1.25)--(8.5,1.5)--(10,0);
 \draw[thick] (2.75,0)--(2.75,1.25);
 \draw[thick] (2.5,1)--(6.5,1);
 \draw[thick] (7.25,0)--(7.25,1.25);
 \draw[thick] (7,1) --(9,1);
 \draw (2.5,.5) node {$h_i$};
 \draw (5,.65) node {$l_i$};
 \draw(7,.5) node {$h_j$};
 \draw (8,.6) node {$l_j$};
 
\end{tikzpicture}
\caption{The $j$th excursion occurs after the $i$th excursion}
\end{figure}

In either case $\sigma(j)-\sigma(i) \neq 0$ so $\sigma$ is in fact a bijection from $[n]$ to $[n].$ 
\end{proof}

Now we show that it is {\bf 231}-avoiding.

\begin{proof}[Proof of Theorem \ref{bijection}]
If $\sigma\notin \sn(231),$ then there exists $i<j<k$ such that $\sigma(k)<\sigma(i)<\sigma(j)$.  Note that $\sigma(k)<\sigma(i)$ implies the $k$th up-step occurs before the end of the $i$th excursion.  By Equations \ref{insideexcursion} and \ref{afterexcursion} the $k$th up-step occurs before the end of the $i$th excursion.  Therefore the $j$th up-step also occurs before the end of the $i$th excursion which implies $\sigma(j)<\sigma(i)$ so $\sigma$ must be {\bf 231}-avoiding.

All that remains is to show that $\sigma_\gamma \neq \sigma_{\gamma'}$ if $\gamma\neq\gamma'.$  The quantity $\sigma_\gamma(i)-i$ attains a local minimum exactly when $l_i(\gamma) = 2$ and $\sigma(i)-i = -h_i +1.$  But $l_i=2$ implies that $2i-h_i$ is a local maximum of the Dyck path.  Hence there is one-to-one correspondence with local minima of $\sigma_\gamma(i) - i$ and local maxima of $\gamma$.   A Dyck path is uniquely defined by the height and location of the local maxima.  Hence the map from $\sigma_\gamma\to\gamma$ is well-defined.  Therefore the map from $\dn$ to $\sn(231)$ given by $(\gamma \to \sigma_\gamma)$ is a bijection.
\end{proof}

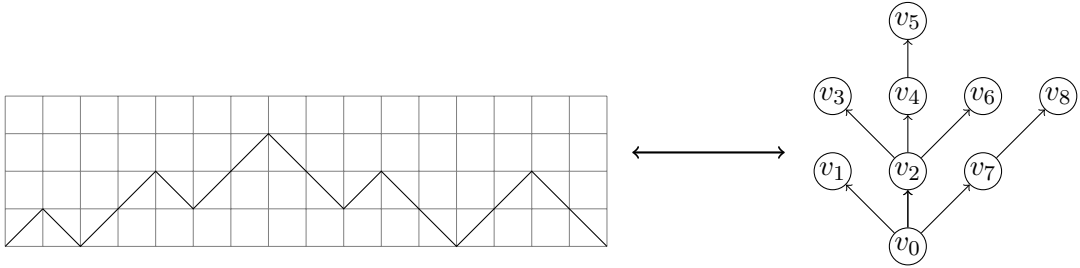
\begin{figure}[h]
\centering
 \begin{tikzpicture} 
 \draw[step=.5, help lines] (0,0) grid (8,2);
 \draw (0,0) --(.5,.5);
 \draw (.5,.5)--(1,0)--(2,1)--(2.5,.5)--(3.5,1.5)--(4.5,.5)--(5,1)--(6,0)--(7,1)--(8,0);

\Vertexn(a) at (8.2,1.25) {};
\Vertexn(b) at (10.5,1.25) {};
 
 \Vertex(r) at (12,0) {$v_0$};
\Vertex (1) at (11,1) {$v_1$};

\Vertex (2) at (12,1) {$v_{2}$};

\Vertex (3) at (13,1) {$v_{7}$};

\Vertex (5) at (11,2) {$v_3$};
\Vertex (6) at (12,2)  {$v_4$};
\Vertex (7) at (13,2) {$v_{6}$};
\Vertex (8) at (14,2) {$v_{8}$};
\Vertex (10) at (12,3) {$v_5$};

\path[every edge, ->] (r) edge (2) (r) edge (1) (r) edge (2) (r) edge (3) (2) edge (5) (2) edge (6) (2) edge (7) (3) edge (8)  (6) edge (10) ;
\path[thick, every edge, <->] (a) edge (b);

\end{tikzpicture}
\caption{A Dyck path and the rooted ordered tree for which it is the contour process.  The root is $v_0$ and the vertices are labeled in order of appearance on the depth-first walk of the tree.}
\label{fig.dycktree}
\end{figure}

\subsection{Connection to rooted ordered tree}

We remark that this bijection can be interpreted in terms of rooted ordered trees by considering a Dyck path of length $2n$ as the contour process of a rooted ordered tree with $n+1$ vertices, as described in Figure \ref{fig.dycktree}.  Given a Dyck path $\gamma$, we denote by $\mathbf{t}^\gamma$ the corresponding rooted ordered tree.  Formally, this bijection is constructed as follows.  Given a tree rooted ordered tree $\mathbf{t}$ with $n+1$ vertices, the depth-first walk of $\mathbf{t}$ is the function $f_{\mathbf{t}} : \{0,1,\dots, 2n\}\to \mathbf{t}$ defined by $f_\bt(0) = \mathrm{root}$ and given $f_\bt(i)=v$, $f_\bt(i+1)$ is the left most child of $v$ that has not already been visited if such a child of $v$ exists, and the parent of $v$ otherwise.  The Dyck path $\gamma$ corresponding to $\bt$ is defined by $\gamma(i) = d(\mathrm{root}, f_\bt(i))$, where the distance between two vertices is the number of edges on the path between them.

As in Figure \ref{fig.dycktree}, we always consider the vertices of a rooted ordered tree to be labeled in their order of appearance on the depth-first walk of the tree with the root labeled $v_0$.  For vertices $v,w\in \bt$, we say $v$ is an ancestor of $w$ (or $w$ is a descendant of $v$) if $v$ is on the path from $w$ to the root of $\bt$. let $\bt_v$ be the fringe subtree of $\bt$ rooted at $v$.  That is, $\bt_v$ is the rooted ordered tree comprised of the vertices $w\in \bt$ such that $v$ is on the path from $w$ to the root.  

This ancestral relationship induces a partial order $\preceq $ on the vertices of $\bt$, which we define by $w \preceq v$ if $v$ is an ancestor of $w$ (we consider $v$ to be an ancestor of itself).  Note that under this order $v$ is always the largest element of $\bt_v$.   If $\gamma$ is a Dyck path, the relative order structure of $\sigma_\gamma$ is completely determined by the order structure on $\bt^\gamma$ in the sense that, if $i<j$ then $\sigma_\gamma(j)< \sigma_\gamma(i)$ if and only if $v_j \preceq v_i$.  Moreover, we can easily express $\sigma_\gamma$ in terms of $\bt^\gamma$ using the formula 
\begin{equation}\label{eq tree bijection} \sigma_\gamma(i) = i + |\bt^\gamma_{v_i}| - \mathrm{ht}_\gamma(v_i), \quad i=1,2,\dots, |\bt^\gamma|,\end{equation}
where $\mathrm{ht}_\gamma(v_i)$ is the height of $v_i \in \bt^\gamma$ (i.e. the number of edges on the path from $v_i$ to the root), and $|\bt|$ is the number of vertices of $\bt$.  Many quantities related to the order structure of $\sigma_\gamma$ can easily be read off of this tree representation.  For example, recall that the path length of a tree is defined by
\[ \textrm{PathLength}(\bt^\gamma) = \sum_{i=1}^{|\bt|} \mathrm{ht}_\gamma(v_j).\]
We can obtain a formula for the number of inversions in $\sigma_\gamma$ in terms of the path length of $\bt^{\gamma}$ as follows:
\[ \#\{ (i,j) : i<j \textrm{ and } \sigma_\gamma(j)<\sigma_\gamma(i)\} = \sum_{j=1}^{|\bt^\gamma|-1} (\mathrm{ht}_\gamma(v_j)-1) =  \textrm{PathLength}(\bt^\gamma) - |\bt^\gamma|+1,\]
since for fixed $j$, the number of $i<j$ such that $\sigma_{\gamma}(j)<\sigma_\gamma(i)$ is equal to the number of vertices on the path from $v_j$ to the root $v_0$, excluding $v_j$ and $v_0$, which is precisely $\mathrm{ht}_{\gamma}(v_j)-1$.  If $\Gamma^n$ is a uniformly random Dyck path of length $2n$, then $\bt^{\Gamma^n}$ is a uniformly random rooted ordered tree with $n+1$ vertices.  It is a well known result of Tak{\'a}cs \cite{Tak91} (and is also an immediate consequence of Proposition \ref{proposition height limit} below) that 
\[ \lim_{n\to\infty}  \frac{\textrm{PathLength}\left(\bt^{\Gamma^n}\right)}{n^{3/2}} =_d \sqrt{2} \int_0^1\mathbbm{e}_t dt,\]
where $(\mathbbm{e}_t,0\leq t\leq 1)$ is Brownian excursion and, consequently,
\[ \lim_{n\to\infty}  \frac{ \#\{ (i,j) : i<j \textrm{ and } \sigma_{\Gamma^n}(j)<\sigma_{\Gamma^n}(i)\} }{n^{3/2}} =_d \sqrt{2} \int_0^1\mathbbm{e}_t dt.\]
This limiting distribution was first observed in \cite{Ja14}, and the computation above fully explains the connection between the asymptotic number of decreases in a $\mathbf{231}$-avoiding permutation and the asymptotic path length of conditioned Galton-Watson trees (since a uniform rooted ordered tree with $n$ vertices is a conditioned Galton-Watson tree) observed in \cite{Ja14}.  Because of how naturally the bijection above relates the order structure of a permutation to the contour process of a ordered tree, we suspect that it can be used to give alternative, possibly simpler, proofs of a number of results in \cite{Ja14}.  We do not pursue this here because it seems to be a relatively straightforward matter of translating the ideas in \cite{Ja14} through the bijection above instead of using the more classical bijection with binary trees used in \cite{Ja14}. 

Another easy consequence of this bijection is the following.
Let $$M^n(\gamma)=\max_{i \in \{0,1,...\,2n\}} \gamma(i)$$ and
let $$m^n(\sigma_\gamma)=\max_{i \in \{0,1,...\,2n\}} i-\sigma_\gamma(i).$$
As the maximum of a Dyck path occurs in an excursion of length 1 so
$$M^n(\gamma)=1+m^n(\sigma_\gamma).$$ This gives us the following.
\begin{lemma}
The distribution of $M^n$ is the same as the distribution of $1+m^n$.
\end{lemma}

\section{Invariance principles for 231-avoiding permutations}
Interpreting $\gamma \in \mathrm{Dyck}^{2n}$ as the contour process of a rooted ordered tree $\bt^\gamma$ with $n+1$ vertices as above, we see that the number of excursions of $\gamma$ of length $2k$ is equal to the number of proper fringe subtrees of $\bt$ with $k$ vertices.  Furthermore from the definition of $\sigma_\gamma$  we have that
\[ i-\sigma_\gamma(i) = \mathrm{ht}_\gamma(v_i)-|\bt^\gamma_{v_i}|.\]
The next proposition shows that $( \mathrm{ht}_\gamma(v_i))_{i=0}^n$ is typically close to $\gamma$.

\begin{proposition}\label{proposition height limit}
Let $\Gamma^n$ be a uniformly random Dyck path of length $2n$.  For every $\epsilon>0$ we have
\[\lim_{n\to\infty} \P\left( \max_{0\leq i\leq n} |\Gamma^n_{2i} -   \mathrm{ht}_{\Gamma^n}(v_i)| >\epsilon \sqrt{n}\right)=0.\]   
\end{proposition}

This result is well known in the field of scaling limits for random trees and, for example, is implicit in \cite{MaMo03}.  We include a proof since we are working in a very special case where the proof is simpler than in the general cases considered in the literature.  Stronger approximation results than this can be obtained, but this proposition is sufficient for our purposes.  Similar arguments can be found in \cite{MaMo03, PiRi13}.

\begin{proof}
Let $\Gamma^n = (\Gamma^n_0,\Gamma^n_1,\dots, \Gamma^{n}_{2n})$ be a uniformly random Dyck path of length $2n$.  Define $V^n = (V^n_0,\dots, V^n_n)$ by $V^n_0=0$ and for $1\leq i\leq n$, define $V^n_i = \inf\{ k> V^n_{i-1} : \Gamma^n_{k}-\Gamma^n_{k-1} =1\}$
Let $S=(S_m,m\geq 0)$ be a simple symmetric random walk on $\Z$ with $S_0=0$.  Define $V_0=0$ and for $m\geq 1$ let $V_m = \inf\{k>V_{m-1} : S_k-S_{k-1} = 1\}$.  Let $\eta(S) = \inf\{ k : S_k=-1\}$.  The bijection we are using between Dyck paths and rooted ordered trees implies that
\begin{equation}\label{eq w-cond} (\Gamma^n, (\mathrm{ht}_{\Gamma^n}(v_i))_{i=0}^n, V^n) =_d\left( \left(S_i)\right)_{i=0}^{2n}, \left(S_{V_i})\right)_{i=0}^{n}, \left(V_i\right)_{i=0}^n\right) \textrm{ given } \eta(S)=2n+1.\end{equation}
By Proposition \ref{prop walk flucs} (taking $\beta=1/2$ and $\alpha = 3/4$) there exist constants $C,D>0$ such that
\begin{equation}\label{eq v-close} \P\left( \max_{0\leq i\leq n} |S_{2i} - S_{V_i} | >\epsilon \sqrt{n}\right)\leq Ce^{-D n^{1/4}} \end{equation}
Combining \eqref{eq v-close} with the fact that $\P(\eta(S)=2n+1) = O(n^{-3/2})$ shows that 
\[\begin{split} \P\left( \max_{0\leq i\leq n} |\Gamma^n_{2i} -   \mathrm{ht}_{\Gamma^n}(v_i)| >\epsilon \sqrt{n}\right) & = P\left( \max_{0\leq i\leq n} |S_{2i} - S_{V_i} | >\epsilon \sqrt{n}\  \Big|  \eta(S)=2n+1 \right)\\
& = O(n^{3/2}e^{-D n^{1/4}}),\end{split}\]
which proves the proposition.
\end{proof}

Since Proposition \ref{proposition height limit} shoes that, appropriately rescaled, $( \mathrm{ht}_{\Gamma^n}(v_{\fl{t}}),0\leq t\leq n)$ converges to Brownian excursion, we see that $( i-\sigma_{\Gamma^n}(i) , 1\leq i\leq n)$ will be well approximated by Brownian excursion at values of $i$ such that $|\bt^{\Gamma^n}_{v_i}|$ is small.  Our next step in establishing invariance principles related to $( i-\sigma_{\Gamma^n}(i) , 1\leq i\leq n)$ is to estimate the expected number of $i$ for which $|\bt^{\Gamma^n}_{v_i}|$ is large. 

 Let $\xi_k(\bt)$ be the number of fringe subtrees of $\bt$ with $k$ vertices.  So long as $k<|\bt|$, $\xi_k(\bt)$ is the number of excursions of $\gamma$ of length $2k$.  
 \begin{lemma}\label{lemma tree counting}
 Let $\bT^n$ be a uniformly random rooted ordered tree with $n+1$ vertices.  For $k \leq n+1$ we then have
 \[ \E \xi_k(\bT^n) = \frac{C_{k-1}}{2 C_n} {{2(n+1-k)} \choose {n+1-k}} + \frac{1}{2} \cf_{\{k=n+1\}}.\]
 Moreover, there exists a function $\Delta$ such that 
 \begin{equation} \label{delt}
\Delta(n) = O(1/n)
\end{equation} as $n\to \infty$ and
 \begin{equation} \label{we are the best} \E \xi_k(\bT^n) = \frac{4^{n+1-k} C_{k-1}}{2C_n\sqrt{\pi(n+1-k)} } \left(1+\Delta(n+1-k)\right) + \frac{1}{2 } \cf_{\{k=n+1\}}.
 \end{equation}
 \end{lemma}

We remark that the asymptotic statement follows directly from the exact formula using the classical asymptotic estimate for central binomial coefficients. 
 
\begin{proof}
Let 
\[ \Xi_k(z) = \sum_{\bt} \xi_k(\bt) z^{|\bt|} \quad \textrm{and} \quad y(z) = \sum_{\bt} z^{|\bt|} = \frac{1-\sqrt{1-4z}}{2}, \]
where $|\bt|$ is the number of vertices of $\bt$ and the sums are over all finite rooted ordered trees.  The computation of $y(z)$ can be found in e.g. \cite[Section I.5.1]{FlSe09}.  If $\bt_1,\dots, \bt_r$ are the fringe subtrees of $\bt$ attached to the root, then 
$$\xi_k(\bt) = \cf\{|\bt|=k\} + \sum_{j=1}^r \xi_k(\bt_j),$$ 
so $\xi_k$ is a recursive additive functional and applying \cite[Lemma VII.1 p.457]{FlSe09} yields
\[ \Xi_k(z) =  \frac{C_{k-1}z^{k+1} y'(z)}{y(z)} = \frac{C_{k-1}}{2} z^k + \frac{C_{k-1}}{2} z^k (1-4z)^{-1/2},\]
where $C_n$ is the $n$th Catalan number.  Recall that $z\mapsto (1-4z)^{-1/2}$ is the generating function for the central binomial coefficients, that is
\[ \frac{1}{\sqrt{1-4z}} = \sum_{j=0}^\infty {{2j} \choose j} z^j,\]
and consequently
\[ \Xi_k(z) = \frac{C_{k-1}}{2} z^k + \sum_{j=k}^\infty \frac{C_{k-1}}{2} {{2(j-k)} \choose {j-k}} z^j.\]
For a power series $f(z)$, let $[z^n]f(z)$ be the coefficient of $z^n$.  The first claim of the lemma follows since
\[\E \xi_k(\bT^n)  = \frac{[z^{n+1}]\Xi_k(z)}{C_n},\]
and the second follows from standard asymptotic estimates of central binomial coefficients.
\end{proof}

For a tree $\bt$, let 
\begin{equation} \label{dubliners}
\hat\xi_k(\bt) = \sum_{j=k}^{|\bt|-1} \xi_j(\bt)
\end{equation}
be the number of proper fringe subtrees of $\bt$ with at least $k$ vertices.

\begin{theorem}\label{theorem subtree sizes}
Let $\bT^n$ be a uniformly random rooted ordered tree with $n+1$ vertices and let $k_n = \fl{cn^\alpha}$.
\begin{enumerate}
\item If $c>0$ and $0<\alpha<1$ then
\[ \lim_{n\to\infty} \frac{1}{n^{1-\alpha/2}} \E  \hat\xi_{k_n}(\bT^n) = \frac{1}{\sqrt{\pi c}}.\]
\item If $0<c<1$ and $\alpha=1$ then
\[ \lim_{n\to\infty} \frac{1}{\sqrt{n}} \E   \hat\xi_{k_n}(\bT^n) =  \sqrt{\frac{1-c}{\pi c}}. \]
\end{enumerate}
\end{theorem}

The proof of this theorem is a rather lengthy computation using Lemma \ref{lemma tree counting} and Stirling's formula.  Our generating function computations in \eqref{we are the best} and the definition of $\hat\xi$ in \eqref{dubliners} show that
\begin{equation} \label{analects} \frac{1}{n^{1-\alpha/2}}\E  \hat\xi_{k_n}(\bT^n) = \frac{4^{n+1}}{2\sqrt{\pi} n^{3/2} C_n} n^{\alpha/2}\sum_{k=\fl{cn^\alpha}}^{n} \frac{C_{k-1}}{4^k\sqrt{1-(k-1)/n}} \left(1 + \Delta\left(n+1-k\right)\right). \end{equation}  
Because $k$ can be of the same order as $n$ in \eqref{analects}, some care is needed to handle the summation and show that $\Delta$ does not impact the results.

As such, we include a sketch of the proof that shows how to handle the sublties that arise and omit the more routine aspects of the calculation.   

\begin{proof}
We first analyze the case where $c>0$ and $0<\alpha <1$.  Because $k$ can be of the same order as $n$ in \eqref{analects}, in order to use our asymptotic knowledge of $\Delta$, we need to trancate the summation.  To do this, we fix $\beta$ such that 
\begin{equation} \label{betachoice}
(2+\alpha)/3 < \beta <1.
\end{equation}  
For large enough $n$, we can break \eqref{analects} up into two parts as follows
\begin{multline}\label{equation beta-decomp}  \frac{1}{n^{1-\alpha/2}}\E  \hat\xi_{k_n}(\bT^n) = \frac{4^{n+1}}{2\sqrt{\pi} n^{3/2} C_n} n^{\alpha/2}\sum_{k=\fl{cn^\alpha}}^{\fl{n^\beta}} \frac{C_{k-1}}{4^k\sqrt{1-(k-1)/n}} \left(1 + \Delta\left(n+1-k\right)\right) \\ 
+ \frac{4^{n+1}}{2\sqrt{\pi} n^{3/2} C_n} n^{\alpha/2}\sum_{k=\fl{n^\beta}+1}^{n} \frac{C_{k-1}}{4^k\sqrt{1-(k-1)/n}} \left(1 + \Delta\left(n+1-k\right)\right).
\end{multline}
We analyze these two terms separately, but first recall that 
\begin{equation} \label{stirling}
\frac{4^n}{\sqrt{\pi} n^{3/2} C_n} \to 1
\end{equation} by Stirling's formula.  We now show that the second term in \eqref{equation beta-decomp} vanishes as $n$ goes to infinity.  It is enough to show that
\[ n^{\alpha/2}\sum_{k=\fl{n^\beta}+1}^{n} \frac{C_{k-1}}{4^k\sqrt{1-(k-1)/n}} \left(1 + \Delta\left(n+1-k\right)\right) \to 0\]
since the factor in front converges by the version of Stirling's formula in \eqref{stirling}.  Since $\Delta(n)= O(1/n)$ by \eqref{delt} there is some constant $B$ such that $|1+\Delta(n)| \leq B$ for all $n$.  Thus  
\begin{multline*} \left|  \sum_{k=\fl{n^\beta}+1}^{n} \frac{C_{k-1}}{4^k\sqrt{1-(k-1)/n}} \left(1 + \Delta\left(n+1-k\right)\right) \right| \\ \leq B\sum_{k=\fl{n^\beta}+1}^{n} \frac{C_{k-1}}{4^k\sqrt{1-(k-1)/n}}  \\
 \leq \frac{B}{n^{3\beta/2}} \sum_{k=\fl{n^\beta}+1}^{n} \frac{C_{k-1}  k^{3/2} }{4^k} \frac{1}{\sqrt{1-(k-1)/n}}.
\end{multline*}
Using the version of Stirling's formula in \eqref{stirling} again, for sufficiently large $n$ we have
\[\begin{split} \frac{B}{n^{3\beta/2}} \sum_{k=\fl{n^\beta}+1}^{n} \frac{C_{k-1}  k^{3/2} }{4^k} \frac{1}{\sqrt{1-(k-1)/n}}& \leq  \frac{B}{n^{3\beta/2}} \sum_{k=\fl{n^\beta}+1}^{n} \frac{1}{\sqrt{1-(k-1)/n}} \\
&  =  Bn^{1-\frac{3\beta}{2}}\left[ \frac{1}{n} \sum_{k=\fl{n^\beta}+1}^{n} \frac{1}{\sqrt{1-(k-1)/n}}\right] \\
& \leq   Bn^{1-\frac{3\beta}{2}} \int_0^1 \frac{1}{\sqrt{1-x}} dx \\
& = 2Bn^{1-\frac{3\beta}{2}}
\end{split}.\]
Since $(2+\alpha)/3 <\beta$  by \eqref{betachoice} we have $1+\frac{\alpha}{2} - \frac{3\beta}{2} <0$, and as a result
\begin{equation} \label{screech}
\left| n^{\alpha/2}\sum_{k=\fl{n^\beta}+1}^{n} \frac{C_{k-1}}{4^k\sqrt{1-(k-1)/n}} \left(1 + \Delta\left(n+1-k\right)\right) \right| \leq2Bn^{1+\frac{\alpha}{2}-\frac{3\beta}{2}} \to 0.\end{equation}
We now turn to the first term in \eqref{equation beta-decomp}.  Since $k\leq \fl{n^\beta}$ and $\beta<1$ we have that 
\begin{equation} \label{kelly}
n+1-k \to \infty \qquad \text{ and } \qquad 1-(k-1)/n \to 1,
\end{equation} both uniformly in $k$.  Thus, given $\epsilon >0$, for sufficiently large $n$, we have by \eqref{kelly} and \eqref{delt}
\[\begin{split}
 (1-\epsilon) n^{\alpha/2} \sum_{k=\fl{cn^\alpha}}^{\fl{n^\beta}} \frac{C_{k-1}}{4^k} & \leq n^{\alpha/2}\sum_{k=\fl{cn^\alpha}}^{\fl{n^\beta}} \frac{C_{k-1}}{4^k\sqrt{1-(k-1)/n}} \left(1 + \Delta\left(n+1-k\right)\right)\\
&\leq   (1+\epsilon) n^{\alpha/2} \sum_{k=\fl{cn^\alpha}}^{\fl{n^\beta}} \frac{C_{k-1}}{4^k}
.\end{split}\]
Applying the version of Stirling's formula in \eqref{stirling} for sufficiently large $n$, we have
\begin{equation} \label{saved by the bell} \begin{split} n^{\alpha/2}\sum_{k=\fl{cn^\alpha}}^{\fl{n^\beta}} \frac{C_{k-1}}{4^k\sqrt{1-(k-1)/n}} \left(1 + \Delta\left(n+1-k\right)\right) & \sim n^{\alpha/2}\sum_{k=\fl{cn^\alpha}}^{\fl{n^\beta}} \frac{C_{k-1}}{4^k}\\
& \sim  \frac{n^{\alpha/2}}{4\sqrt{\pi}}  \sum_{k=\fl{cn^\alpha}}^{\fl{n^\beta}} \frac{1}{k^{3/2} }\\
& \rightarrow \frac{1}{2\sqrt{\pi c}}.\end{split}\end{equation}
Combining the computations in \eqref{equation beta-decomp}, \eqref{stirling}, \eqref{screech} and \eqref{saved by the bell} we have
\[ \lim_{n\to\infty} \frac{1}{n^{1-\alpha/2}} \E  \hat\xi_{k_n}(\bT^n) = \frac{1}{\sqrt{\pi c}},\]
as desired.

We now turn to the case when $0<c<1$ and $\alpha =1$.  Fix some $c<d<1$, so that breaking \eqref{analects} up into two parts we get
\begin{multline}\label{equation decomp 2} \frac{1}{\sqrt{n}}\E \hat\xi_{k_n}(\bT^n) = \frac{4^{n+1}}{2\sqrt{\pi} n^{3/2} C_n} n^{1/2}\sum_{k=\fl{cn}}^{\fl{dn}} \frac{C_{k-1}}{4^k\sqrt{1-(k-1)/n}} \left(1 + \Delta\left(n+1-k\right)\right) \\
+ \frac{4^{n+1}}{2\sqrt{\pi} n^{3/2} C_n} n^{1/2}\sum_{k=\fl{dn}+1}^{n} \frac{C_{k-1}}{4^k\sqrt{1-(k-1)/n}} \left(1 + \Delta\left(n+1-k\right)\right)
.\end{multline}
Looking at the second term in \eqref{equation decomp 2}, using the version of Stirling's formula in \eqref{stirling} and the bound on $\Delta(n)$ in \eqref{delt} we see that there exists some constant $D$ such that
\[ \left| n^{1/2}\sum_{k=\fl{dn}+1}^{n} \frac{C_{k-1}}{4^k\sqrt{1-(k-1)/n}} \left(1 + \Delta\left(n+1-k\right)\right) \right| \leq \frac{D}{n}\sum_{k=\fl{dn}+1}^{n} \frac{1}{\sqrt{1-(k-1)/n}}. \]
Consequently 
\begin{multline} \limsup_{n\to\infty} \left| n^{1/2}\sum_{k=\fl{dn}+1}^{n} \frac{C_{k-1}}{4^k\sqrt{1-(k-1)/n}} \left(1 + \Delta\left(n+1-k\right)\right) \right|\\
\leq D' \int_{d}^1 \frac{1}{\sqrt{1-x}} dx 
= 2 D' \sqrt{1-d}, \label{belding}
\end{multline}
which can be made arbitrarily small depending on our choice of $d$.  

We now turn our attention to the first term of \eqref{equation decomp 2}. Since $n+1-k \to \infty$, given $\epsilon>0$, for sufficiently large $n$ we can use the version of Stirling's formula in \eqref{stirling} to get 
\begin{multline} \label{zack}
\frac{(1-\epsilon)}{4\sqrt{\pi}n} \sum_{k=\fl{cn}}^{\fl{dn}} \frac{1}{ \left(\frac{k-1}{n}\right)^{3/2} \sqrt{1-(k-1)/n}} \\
\leq 
n^{1/2}\sum_{k=\fl{cn}}^{\fl{dn}} \frac{C_{k-1}}{4^k\sqrt{1-(k-1)/n}} \left(1 + \Delta\left(n+1-k\right)\right) \\
\leq \frac{(1+\epsilon)}{4\sqrt{\pi}n} \sum_{k=\fl{cn}}^{\fl{dn}} \frac{1}{ \left(\frac{k-1}{n}\right)^{3/2} \sqrt{1-(k-1)/n}}.
\end{multline}
Furthermore,
\begin{eqnarray}\nonumber 
\lim_{n\to\infty} \frac{1}{n} \sum_{k=\fl{cn}}^{\fl{dn}} \frac{1}{ \left(\frac{k-1}{n}\right)^{3/2} \sqrt{1-(k-1)/n}} &=& \int_c^d \frac{1}{x^{3/2}\sqrt{1-x}} dx\\ &=& 2\left( \sqrt{c^{-1}-1} - \sqrt{d^{-1}-1}\right).\label{jesse}
\end{eqnarray}
so both the first and last terms in \eqref{zack} are converging to 
\begin{equation}
\label{principal} \frac{1}{2\sqrt{\pi}}\left( \sqrt{c^{-1}-1} - \sqrt{d^{-1}-1}\right).
\end{equation}
Combining  the limits for \eqref{belding} and the middle term in \eqref{zack}  and picking $d$ arbitrarily close to 1  we find that
\begin{equation}\label{slater} \lim_{n\to\infty}  n^{1/2}\sum_{k=\fl{cn}}^{n} \frac{C_{k-1}}{4^k\sqrt{1-(k-1)/n}} \left(1 + \Delta\left(n+1-k\right)\right) =  \sqrt{\frac{1-c}{4\pi c}}.\end{equation}
Plugging \eqref{slater} into \eqref{analects} and using \eqref{stirling} we get
\[ \lim_{n\to\infty} \frac{1}{\sqrt{n}} \E  \hat\xi_{k_n}(\bT^n)=  \sqrt{\frac{1-c}{\pi c}}, \]
as desired.
\end{proof}

\begin{theorem}\label{theorem excluded invariance}
Let $\Gamma^n$ be a uniformly random Dyck path of length $2n$ and $\sigma_{\Gamma^n}$ the corresponding \textbf{231}-avoiding permutation constructed as above.  For $c,\alpha > 0$, define 
\[B_{c,\alpha,n} = \{i : |\bt^{\Gamma^n}_{v_i}| \leq  cn^\alpha\}.\] 
If $0<\alpha<1/2$, then for every $\epsilon >0$ we have
\[ \lim_{n\to \infty} \P\left( \sup_{0\leq t\leq 1} \left|\frac{1}{\sqrt{2n}} \Gamma^n(2nt) + F^{B_{c,\alpha,n}}_{\sigma_{\Gamma^n}}(t) \right| > \epsilon \right) =0.\]
\end{theorem}

\begin{proof}
Observe that
\[\begin{split} \max_{i \in B_{c,\alpha,n}} \left| \frac{1}{\sqrt{2n}} \Gamma^n(2i) + F^{B_{c,\alpha,n}}_{\sigma_{\Gamma^n}}(i/n)\right| & = \frac{1}{\sqrt{2n}} \max_{i \in B_{c,\alpha,n}}\left| \Gamma^n(2i) - \Ht_{\Gamma^n}(v_i) + |\bt^{\Gamma^n}_{v_i}|\right| \\
& \leq \frac{1}{\sqrt{2n}} \max_{i \in B_{c,\alpha,n}}\left| \Gamma^n(2i) - \Ht_{\Gamma^n}(v_i)\right| + \frac{c}{\sqrt{2}} n^{\alpha - 1/2}
.\end{split}\]

Since $\alpha <1/2$, this goes to $0$ in probability as $n$ goes to $\infty$ by Proposition \ref{proposition height limit}.  Moreover, by Theorem \ref{theorem subtree sizes} we see that for every $\beta < \alpha/2$ and 
\begin{equation} \label{the college years}
\P\left(n-|B_{c,\alpha,n}| > n^{1-\beta} \right) \leq \frac{1}{n^{1-\beta}} \E \hat\xi_{\fl{cn^{\alpha/2}}}(\bT^n)  \to 0.\end{equation}
An immediate consequence of this is that
\begin{equation}\label{the reunion}  \sup_{0\leq t \leq 1} \min_{ i\in B_{c,\alpha,n}} | t- i/n| \to 0\end{equation}
in probability.  The result now follows from the convergence of Dyck paths to Brownian excursion and the continuity of Brownian excursion.
\end{proof}

\begin{proof}[Proof of Theorem \ref{london}]
The set $\shortexc$ in Theorem \ref{london} is $B_{c,\alpha,n}$ in the above theorem. The claimed bound on the size of $\shortexc$ is given in \eqref{the college years}.
\end{proof}

Finding the set $B_{c,\alpha,n}$ requires some large scale knowledge of the values of the permutation.   We can also prove an invariance principle for randomly selected values.

\begin{theorem}
Let $\Gamma^n$ be a uniformly random Dyck path of length $2n$ and $\sigma_{\Gamma^n}$ the corresponding \textbf{231}-avoiding permutation constructed as above.  Fix $0<\alpha< 1/4$ and let $c_n\uparrow \infty$ be a sequence such that $n^{-\alpha}c_n \to c >0$.  Let $U^n_1, \dots, U^n_{c_n}$ be i.i.d. uniform on $\{1,\dots, n\}$, independent of $\Gamma^n$.  Defining $ B_n = \{U^n_i\}_{i=1}^n$, for every $\epsilon >0$ we have
\[ \lim_{n\to \infty} \P\left( \sup_{0\leq t\leq 1} \left|\frac{1}{\sqrt{2n}} \Gamma^n(2nt) + F^{B_n}_{\sigma_{\Gamma^n}}(t) \right| > \epsilon \right) =0.\]
\end{theorem}

\begin{proof}
The proof is essentially the same as that of Theorem \ref{theorem excluded invariance}.  By Theorem \ref{theorem subtree sizes} for every $d>0$ and $2\alpha<\beta < 1/2$, we have
\[\P\left( \max_{1\leq i \leq c_n} |\bt^{\Gamma^n}_{v_{U^n_i}}|  \geq \fl{dn^\beta}\right) \leq \frac{c_n}{n} \E \hat\xi_{\fl{dn^\beta}}(\bT^n) = \frac{c_n}{n^{\beta/2}}\frac{1}{n^{1-\beta/2}} \E \hat\xi_{\fl{dn^\beta}}(\bT^n) \to 0. \]
Consequently,  
\[ \max_{i \in B_{n}} \left| \frac{1}{\sqrt{2n}} \Gamma^n(2i) + F^{B_n}_{\sigma_{\Gamma^n}}(i/n)\right| \leq \frac{1}{\sqrt{2n}} \max_{i \in B_{n}}\left| \Gamma^n(2i) - \Ht_{\Gamma^n}(v_i) \right|+  \frac{1}{\sqrt{2n}} \max_{i \in B_{n}} |\bt^{\Gamma^n}_{v_i}|\]
goes to $0$ in probability.  In place of \eqref{the reunion} we use the elementary fact that
\[  \sup_{0\leq t \leq 1} \min_{1\leq i \leq c_n } | t- U^n_i/n| \to 0 \]
in probability, and the proof concludes in the same fashion as the proof of Theorem \ref{theorem excluded invariance}. 
\end{proof}

\section{Appendix A: Moderate deviations for random walks}
\label{appendixa}
In this section we recall some classical moderate deviations bounds for random walks.  
The next two results are special cases of \cite[Theorem III.12, Theorem III.15]{Petrov1} respectively (see also \cite[Lemma A1, Lemma A2]{MaMo03}). 

\begin{lemma}\label{lemma maximal inequality}
Let $X_1,X_2,\dots$ be i.i.d with $\E X_1=0$ and let $S_n =X_1+\cdots + X_n$. Suppose that $\sigma^2= \E(X_1^2) < \infty$.  For all $x$ and $n$ we have
\[\P\left(\max_{1\leq k\leq n} S_k \geq x\right) \leq 2 \P\left(S_n\geq x - \sqrt{2n\sigma^2}\right).\]
\end{lemma}

\begin{lemma}\label{lemma moderate deviations}
Let $X_1,X_2,\dots$ be i.i.d with $\E X_1=0$ and let $S_n =X_1+\cdots + X_n$.  Suppose that there exists $a>0$ such that $\E(e^{t |X_1|}) < \infty$.  Then there exist constants $g, T>0$, independent of $n$, such that
\[ \P(S_n\geq x) \leq \begin{cases} \exp\left(-\frac{x^2}{2gn}\right) &  \textrm{if } 0\leq x\leq ngT \\ \\ \exp\left(-\frac{Tx}{2}\right) & \textrm{if }  x\geq ngT.\end{cases}\]
\end{lemma}

These lemmas lead immediately to the following corollary. 

\begin{corollary}\label{corollary fluctuations}
Maintaining the hypotheses of Lemma \ref{lemma moderate deviations}, fix $\epsilon, c>0$ and $0< \alpha< 2\beta$ and let $\nu  = \min(\beta, 2\beta-\alpha)$.  There exist constants $A,B >0$ such that
\[ \P\left( \max_{1\leq i\leq n} \max_{|i-j|\leq cn^{\alpha}} |S_j -S_i| \geq \epsilon n^\beta \right)\leq A \exp\left( -B n^{\nu}\right)\]
\end{corollary}

Let $S=(S_m,m\geq 0)$ be a simple symmetric random walk on $\Z$ with $S_0=0$.  Define $V_0=0$ and for $m\geq 1$ let $V_m = \inf\{k>V_{m-1} : S_k-S_{k-1} = 1\}$.  Let $\eta(S) = \inf\{ k : S_k=-1\}$.  Observe that $(V_m-V_{m-1})_{m\geq 1}$ is an i.i.d sequence of geometric random variables with parameter $1/2$.

\begin{corollary}\label{corollary V-shift}
Fix $\epsilon >0$ and $1/2 < \alpha\leq 1$.  There exist constants $A, B>0$ such that 
\[ \P\left( \max_{1\leq i\leq n} |V_i - 2i| \geq \epsilon n^\alpha\right) \leq A \exp\left(-B n^{2 \alpha -1}\right).\]
\end{corollary}

\begin{corollary} \label{corollary p1}
Let $\Gamma^n$ be a uniformly random Dyck path of length $2n$.  There exist constants $A,B,\nu>0$ such that for all $n\geq 1$
\[ \P(\max_{1\leq i\leq 2n} \Gamma^n_i \geq 0.4 n^{0.6}) \leq A\exp\left(-Bn^\nu\right) \]
and 
\[ \P\left( \max_{1\leq i\leq 2n} \max_{|i-j|\leq 2n^{0.6}} |\Gamma^n_j -\Gamma^n_i| \geq  0.5 n^{0.4} \right)\leq A \exp\left( -B n^{\nu}\right)\]
\end{corollary}

\begin{proof}
Noting that 
\[ \Gamma^n \overset{d}{=} (S_k, 0\leq k \leq 2n) \textrm{ given } \eta(S) =2n+1,\]
the first claim is an immediate consequence of Lemmas \ref{lemma maximal inequality} and \ref{lemma moderate deviations} combined with the fact that $\P(\eta(S) =2n+1) \sim cn^{-3/2}$ for some $c>0$.  The second claim follows similarly from Corollary \ref{corollary fluctuations}
\end{proof}

Combining these two corollaries, we obtain the following proposition.

\begin{proposition}\label{prop walk flucs}
Fix $\epsilon >0$ and $1/4 < \beta\leq 1$.  For every $\alpha$ such that $1/2 <\alpha < \min(2\beta,1)$ there exist constants $A, B>0$ such that
\[ \P\left( \max_{1\leq i\leq n} |S_{2i} -S_{V_i}| \geq \epsilon n^\beta\right) \leq A \exp\left( -B n^{\nu}\right), \]
where $\nu = \min(\beta, 2\beta-\alpha, 2\alpha -1)$
\end{proposition}

\begin{proof}
Fix $\alpha$ such that $1/2 <\alpha < \min(2\beta,1)$.  By Corollaries \ref{corollary fluctuations} and \ref{corollary V-shift} there exist constants $C, D>0$ such that
\[ \begin{split} \P\left( \max_{1\leq i\leq n} |S_{2i} -S_{V_i}| \geq \epsilon n^\beta\right) & \leq \P\left( \max_{1\leq i\leq n} |V_i - 2i| \geq  n^\alpha \right) +  \P\left( \max_{1\leq i\leq n} \max_{|i-j|\leq  n^{\alpha}} |S_j -S_i| \geq \epsilon n^\beta \right) \\
& \leq C\left[ \exp\left(-D n^{2 \alpha -1}\right) + \exp\left( -D n^{\min(\beta, 2\beta-\alpha)}\right) \right] \\
& \leq 2C  \exp\left( -D n^{\nu}\right)
.\end{split}\]
\end{proof}

Let $\kappa_1 = \inf\{ k : S_k -S_{k-1} =-1\}$, and for $i\geq 1$ define
\[  \rho_i =  \inf\{ k>\kappa_i : S_k -S_{k-1} =1\} \quad \textrm{and}\quad \kappa_{i+1} =  \inf\{ k>\rho_i : S_k -S_{k-1} =-1\} .\]
Observe that the number of decreases in the $i$'th run of decreases is $\bar d_i=\rho_i-\kappa_i$.  Observe that $(d_i)_{i=1}^\infty$ is an i.i.d. sequence of geometric random variables with $\P(\bar d_1=k) = 1/2^{k}$ for $k=1,2,\dots$.  Let $\bar D_k=\sum_{i=1}^k \bar d_i$.

\begin{proposition}\label{proposition random walk decreases}
For every $\epsilon>0$ there exist constants $A,B,\nu>0$ such that
\[ \P\left( \max_{1\leq i,j \leq n \atop |i-j|\geq n^{0.3}} |\bar D_j -\bar D_i - 2(j-i)| \geq  \epsilon |i-j|^{0.6} \right)\leq A \exp\left( -B n^{\nu}\right)\]
\end{proposition}

\begin{proof}
Applying the union bound and the stationarity of random walk increments followed by Lemma \ref{lemma moderate deviations}, we find that there exist $c_1, c_2>0$ such that
\[\begin{split}  \P\left( \max_{1\leq i,j \leq n \atop |i-j|\geq n^{0.3}} |\bar D_j -\bar D_i - 2(j-i)| \geq  \epsilon |i-j|^{0.6} \right)& \leq n \sum_{k=\fl{n^{0.3}}}^{n}  \P\left( |\bar D_k-2k| \geq  \epsilon k^{0.6} \right)\\
& \leq n \sum_{k=\fl{n^{0.3}}}^{n} c_1 e^{- c_2 k^{0.2}} \\
& \leq c_1 n^2 e^{- c_2 n^{0.06}} 
,\end{split}\]
and the lemma follows. 
\end{proof}

\begin{corollary}\label{corollary decrease bound}
Let $\Gamma^n$ be a uniformly random Dyck path of length $2n$ and let $\zeta(\Gamma^n)$ be the number of runs of decreases in $\Gamma^n$ (which also equals the number of runs of increases).  For $k\leq \zeta(\Gamma^n)$, let $D_k$ ($A_k$) be the number of decreases contained in the first $k$ runs of decreases (increases).  There exist constants $A,B,\nu >0$ such that
\[ \P\left( \max_{1\leq i,j \leq \zeta(\Gamma^n) \atop |i-j|\geq n^{0.3}} | D_j -D_i - 2(j-i)| \geq  .1|i-j|^{0.6} \right)\leq A \exp\left( -B n^{\nu}\right)\]
and 
\[ \P\left( \max_{1\leq i,j \leq \zeta(\Gamma^n) \atop |i-j|\geq n^{0.3}} | A_j -A_i - 2(j-i)| \geq  .1|i-j|^{0.6} \right)\leq A \exp\left( -B n^{\nu}\right)\]
\end{corollary}

\begin{proof}
The reversibility of uniformly random Dyck paths implies that the claims about decreases and increases are equivalent, so we prove the former.  Let $S=(S_k)_{k\geq 0}$ be a simple symmetric random walk started from $0$ and construct $\bar D=(\bar D_k)_{k\geq 1}$ from $S$ as above.  Let $\mathscr{A}_n$ be the event that $S_{2n}=0$, $S_{2n+1}=1$, and $S_k \geq 0$ for $0\leq k\leq 2n$. Additionally, define $\zeta_n(\bar D) = \inf\{k : \bar D_k \geq n\}$.  Observe that
\[ (D_1,\dots, D_{\zeta(\Gamma^n)}) \overset{d}{=} (\bar D_1,\dots, \bar D_{\zeta_n(\bar D)} ) \textrm{ given } \sA_n .\]
Since $\zeta_n(\bar D)\leq n$, it follows from Proposition \ref{proposition random walk decreases} that
\begin{multline*} \P\left( \max_{1\leq i,j \leq \zeta_n(\bar D) \atop |i-j|\geq n^{0.3}} |\bar D_j -\bar D_i - 2(j-i)| \geq  .1|i-j|^{0.6} \ \Big| \ \sA_n \right) \\ \leq  \P\left( \max_{1\leq i,j \leq n \atop |i-j|\geq n^{0.3}} | \bar D_j -\bar D_i - 2(j-i)| \geq  .1|i-j|^{0.6} \ \Big| \ \sA_n \right) \\
\leq A n^{3/2} \exp\left( -B n^{\nu}\right)
,\end{multline*}
since $\P(\sA_n) \sim c n^{-3/2}$ for some $c>0$.
\end{proof}

\section*{Acknowledgement}
We would like to thank Lerna Pehlivan for many helpful suggestions.  We would also like to thank Igor Pak for a helpful conversation.

\bibliographystyle{plain}

\end{document}